\documentclass[12pt]{article}
\usepackage{amsfonts}
\usepackage{amssymb, amscd, amsthm}
\usepackage{latexsym}
\usepackage{multirow}
\usepackage{amsmath}
\usepackage[all]{xy}
\usepackage[dvips]{graphicx}
\usepackage{mathrsfs}
\usepackage{fancyhdr}

\newtheorem{lemma}[subsection]{Lemma}
\newtheorem{thm}[subsection]{Theorem}
\newtheorem{prop}[subsection]{Proposition}
\newtheorem{rem}[subsection]{Remark}
\newtheorem{coro}[subsection]{Corollary}
\newtheorem{defn}[subsection]{Definition}
\newtheorem{example}[subsection]{Example}

\newcommand{\ra}{\rightarrow}
\newcommand{\mo}{\mathcal{O}}
\newcommand{\mf}{\mathcal{F}}

\newcommand{\me}{\mathcal{E}}
\newcommand{\mext}{\mathbb{E}\mathsf{x}\mathsf{t}}
\newcommand{\mhom}{\mathbb{H}\mathsf{o}\mathsf{m}}
\newcommand{\ms}{\mathcal{S}}
\newcommand{\mt}{\mathcal{T}}
\newcommand{\mh}{\mathcal{H}}
\newcommand{\mm}{\mathcal{M}}
\newcommand{\mn}{\mathcal{N}}
\newcommand{\mw}{\mathcal{W}}
\newcommand{\mv}{\mathcal{V}}
\newcommand{\um}{\mathcal{U}}
\newcommand{\ts}{\mathtt{S}}
\newcommand{\mc}{\mathcal{C}}
\newcommand{\mr}{\mathcal{R}}

\newcommand{\ls}{|L|}
\newcommand{\p}{\mathbb{P}}

\newcommand{\bl}{\mathbb{L}}

\begin{document}
\fontsize{12pt}{14pt} \textwidth=14cm \textheight=21 cm
\numberwithin{equation}{section}
\title{Motivic measures of moduli spaces of 1-dimensional sheaves on rational surfaces.}
\author{Yao YUAN}
\date{}
\maketitle
\begin{flushleft}{\textbf{Abstract.}} We study the moduli space of rank 0 semistable sheaves on some rational surfaces.  We show the irreducibility and stable rationality of them under some conditions.  We also compute several (virtual) Betti numbers of those moduli spaces by computing their motivic measures.

\end{flushleft}

\section{Introduction.}
Let $X$ be a projective rational smooth surface over $\mathbb{C}$, with its canonical bundle $K_X$.  Let $L$ be an effective non trivial line bundle on $X$ and $\chi$ is an integer.  Let $M^{ss}(L,\chi)$ be the (coarse) moduli space of semistable sheaves of rank 0, determinant $L$ and Euler characteristic $\chi$, with respect to some polarization $\mo_X(1)$.  Sheaves in $M^{ss}(L,\chi)$ have Hilbert polynomial $P(n)=L.\mo_{X}(1)n+\chi$, with $L.\mo_{X}(1)$ the intersection number of $L$ and ample line bundle $\mo_X(1)$.  Let $M(L,\chi)$ be the subspace of $M^{ss}(L,\chi)$ parametrizing stable sheaves.

Under some suitable assumption on $L$ and $K_X$, we show the irreducibility of $M^{ss}(L,\chi)$ which generalizes Le Potier's result for $X=\p^2$ (Theorem 3.1 in \cite{lee}).  If moreover there exists a universal sheaf on some open subset of $M(L,\chi)$, we show that then $M(L,\chi)$ is stably rational, hence so is $M^{ss}(L,\chi)$, more precisely $M(L,\chi)\times\p^m$ is rational for some $m$.

Topological invariants of $M^{ss}(L,\chi)$ are of great interests.  For instance, the Euler number $e(M^{ss}(L,\chi))$ is related to the BPS counting in Physics on the local 3-fold associated to $X$.  Although some physicists have computed $e(M^{ss}(L,\chi))$ for a large number of cases on $\p^2$ and $\p^1\times\p^1$ (see Section 8.3 in \cite{kkv}), their argument was not mathematically correct.  In mathematics we only know very few cases (see \cite{yth}) for rational surfaces, while for a K3 or abelian surface,  the deformation equivalence classes of $M^{ss}(L,\chi)$ are known in a large generality by Yoshioka's work in \cite{ky}.

$M^{ss}(L,\chi)$ is also closely related to Pandharipande-Thomas theory defined in \cite{pt} on local 3-folds.  Toda's work in \cite{toda} gives a prediction that $e(M^{ss}(L,\chi))$ does not depend on $\chi$ as long as the whole moduli space is smooth.  In this paper we are not able to prove the prediction but we compute some Betti numbers of $M^{ss}(L,\chi)$ with $X=\p^2$ or a Hirzebruch surfaces and show that they are independent of $\chi$.  For instance we prove the following theorem.

\begin{thm}[Theorem \ref{proj}]\label{in} Let $X=\p^2$ with $H$ the hyperplane class.  Let $b_i$ be the $i$-th Betti number of $M^{ss}(dH,\chi)$.  If $d\geq 8$ and $M^{ss}(dH,\chi)$ is smooth, then we have

(1) $b_0=1,~b_2=2,~b_4=6,~b_6=13,~b_8=29,~b_{10}=57,~b_{12}=113;$

(2) $b_{2i-1}=0$ for $i\leq7$;

(3) For $p+q\leq13$, $h^{p,q}=b_{p+q}\cdot\delta_{p,q}$, where $\delta_{p,q}=\left\{\begin{array}{l}1,~for~p=q.\\ \\ 0,~otherwise.\end{array}\right..$
\end{thm}

Notice that by \cite{mark}, $M^{ss}(L,\chi)$ has all odd Betti numbers zero if it is smooth with a universal sheaf, hence $M^{ss}(L,\chi)=M(L,\chi)$ in this case.  In Theorem \ref{in}, $M(dH,\chi)$ has a universal sheaf if and only if $d,\chi$ are coprime (Theorem 3.19 in \cite{lee}), i.e. $M(dH,\chi)=M^{ss}(dH,\chi)$.  By Theorem \ref{in} we see that the first 13 Betti numbers do not depend on $\chi$, even not on $d$, as long as the moduli space is smooth.  We will see in Section 6 that if $d$ is a prime number or 2 times a prime number, then the first $2d-3$ Betti numbers can be given explicitly and they don't depend on $\chi$.  We also will prove in Section 5 some analogous result to Theorem \ref{in} for $X$ a rational surface.  Although both $M^{ss}(L,\chi)$ and $M(L,\chi)$ depend on the choice of polarization in general, our final result does not and hence we don't mention explicitly the polarization when we talk about those moduli spaces.
  
This is our strategy: choose $\chi<0$, then every 1-dimensional sheaf $F$ with Euler characteristic $\chi$, determinant $L$ can be written into the following non split exact sequence.
\begin{equation}\label{too}0\ra K_X\ra \widetilde{I}\ra F\ra0.\end{equation}
Denote by $g_L$ the arithmetic genus of curves in $\ls$.  If $\widetilde{I}$ is torsion free, then $\widetilde{I}\cong I_{n}(L+K_X)$ with $I_{n}$ an ideal sheaf of colength $n:=g_L-1-\chi$, then we get an element in the Hilbert scheme $Hilb^{[n]}(X)$ of $n$-points on $X$.  However, if $Supp(F)$ is not integral, $\widetilde{I}$ can contain torsion.  Also $F$ in (\ref{too}) with $\widetilde{I}$ torsion free is not necessarily semistable.  In fact, (\ref{too}) provides a biraitonal correspondence between $\text{Ext}^1(F,K_X)$ with $F\in M(L,\chi)$ and $\text{Hom}(K_X,I_{n}(L+K_X))$ with $I_{n}\in Hilb^{[n]}(X)$.  We hence need to estimate the codimensions of the subsets where (\ref{too}) fails to give a correspondence on both sides. 

On the other hand, in general neither $\text{Ext}^1(F,K_X)$ nor $\text{Hom}(K_X,I_{\ell}(L+K_X))$ is of constant dimension over the underlying moduli spaces.  Hence we also need to estimate the codimensions of the subsets where the dimensions of those two spaces jump.  

Instead of working on moduli schemes $M(L,\chi)$ and $Hilb^{[n]}(X)$, most of time we work on moduli stacks $\mm(L,\chi)$ and $\mh^{n}$, where $\mh^{n}$ is viewed as a moduli stack of rank 1 sheaves.  This is because stack language behaves better in dimension estimate and it also allows one to embed the moduli space $\mm(L,\chi)$ into a enlarged space (also a stack) which will contain all $F$ obtained by (\ref{too}), while one can not do this at the scheme level.  Our argument is generally standard, but Section 4 and the appendix are quite technical, where we deal with sheaves with non-reduced but irreducible supports.

The structure of the paper is as follows.  In Section 2, we introduce the enlarged space $\mm^a_{\bullet}(L,\chi)$ containing the moduli stack $\mm(L,\chi)$, and do the dimension estimate of $\mm^a_{\bullet}(L,\chi)-\mm(L,\chi)$.  In Section 3 we study the irreducibility of the moduli space $M^{ss}(L,\chi)$ when there is no sheaf with support non-reduced and irreducible.  Section 4 is the most difficult and complicated part of the paper, where we study the sheaves with support $nC$ for some integral curve $C$ and estimate the dimension of the substack parametrizing those sheaves.  In Section 5, we prove our main result on the motivic measure of the moduli space and also some corollaries.  In Section 6, we list some special results on $\p^2$.  In the end, there is the appendix where we give a whole proof of an important theorem (Theorem \ref{tt}) in Section 4.

\textbf{Acknowledgements.}  I was supported by NSFC grant 11301292.  I thank Yi Hu for some helpful discussions.  I also thank Shenghao Sun for the help on stack theory.    

\section{Some stacks and dimension estimate.}
We fix $X$ to be a projective rational smooth surface over $\mathbb{C}$, with $K_X$ its canonical bundle.   Let $L$ be an effective non trivial line bundle on $X$.  We first introduce some notations and definitions.

\emph{Notations.}

(1) For a sheaf $F$, we denote by $c_1(F)$ the first Chern class of $F$ and $\chi(F)$ the Euler characteristic of $F$.  Define $h^i(F):=dim~H^i(F)$.

(2) Let $C$ be a curve on a surface $X$.  Let $F$ be a sheaf over $X$.  Then $F(\pm C):=F\otimes\mo_X(\pm C)$.  

(3) For two sheaves $F_1,~F_2$ over $X$, $\chi(F_2,F_1):=\sum_{i} (-1)^i dim~\text{Ext}^i(F_2,F_1).$

(4) For two line bundles $L_1$, $L_2$, we write $L_1\leq L_2$ if $L_2-L_1$ is effective.  Write $L_1<L_2$ if $L_1\leq L_2$ and $L_1\neq L_2$.  Write $L>(\geq)~0$ if $L>(\geq)~\mo_X$.  We denote by $L_1.L_2$ the intersection number of their corresponding divisor classes and $L_1^2=L_1.L_1.$
 
(5) Denote by $\ls$ the linear system of  $L$, i.e. $\ls=\p(H^0(L))$ and $\ls^{int}$ the open subset of $\ls$ consisting of integral divisors.  Denote by $g_L$ the arithmetic genus of curves in $\ls$.


\begin{defn}\label{look}We say that $L$ is \textbf{$K_X$-negative} if $\forall~ 0<L'\leq L$, $K_X.L'<0$.
\end{defn}
\begin{rem}\label{neg}If $X$ is Fano,  then any $L$ non trivial and effective is $K_X$-negative. 
\end{rem}
\begin{rem}\label{sm}If $L$ is $K_X$-negative, then $M(L,\chi)$ is either empty or smooth of dimension $L^2+1$. 
\end{rem}
We now define some stacks.  As we said in the introduction, we mainly will work on stacks although our final result is on schemes. 
\begin{defn}\label{ff}Given two integers $\chi$ and $a$, 
let $\mm_{\bullet}^a(L,\chi)$ be the (Artin) stack parametrizing pure sheaves $F$ on $X$ with rank 0, $c_1(F)=L$, $\chi(F)=\chi$ and satisfying either of the following two conditions.
 
($C_1$) $\forall F'\subset F$, $\chi(F')\leq a$;

($C_2$) $F$ is semistable.
\end{defn}
\begin{defn}Let $\mm^{ss}(L,\chi)$ ($\mm(L,\chi)$, resp.) be the substack of $\mm_{\bullet}^a(L,\chi)$ parametrizing semistable (stable, resp.) sheaves in $\mm_{\bullet}^a(L,\chi)$.
\end{defn}
\begin{rem}\label{con}(1) In Definition \ref{ff}, under some suitable assumption on $a$, $\chi$ and $L$, ($C_2$) implies ($C_1$).  But we put ($C_1$) and ($C_2$) together for larger generality.

(2) $\mm(L,\chi)$ has a (coarse) moduli space $M(L,\chi)$.  If we are on $\p^2$, then $M(dH,\chi)$ is a fine moduli space iff $d$ and $\chi$ are coprime (Theorem 3.19 in \cite{lee}).

(3) If $L$ is $K_X$-negative and $M(L,\chi)$ is non-empty, then by Remark \ref{sm} $\mm(d,\chi)$ is of dimension $L^2$.  
\end{rem}

It is easy to see the boundedness of $\mm^a_{\bullet}(L,\chi)$.  Let $\ts^a(L,\chi):=\mm_{\bullet}^a(d,\chi)-\mm(d,\chi)$.  We then estimate the dimension of $\ts^a(L,\chi)$ for $L$ $K_X$-negative.

Define
\begin{equation}\label{dsl}s_L:=\min_{\footnotesize{\begin{array}{c}\{L_k\}_k;\\ \sum_k L_k=L;\\ \forall k, 0<L_k<L\end{array}}}\displaystyle{\sum_{i<j}}L_i.L_j=\frac12(L^2-\max_{\footnotesize{\begin{array}{c}\{L_k\}_k;\\ \sum_k L_k=L;\\ \forall k, 0<L_k<L\end{array}}}\sum_k L^2_k).
\end{equation}

\begin{prop}\label{dlt}Let $L$ be $K_X$-negative, then $dim~\ts^a(L,\chi)\leq L^2-s_L$.
\end{prop}
\begin{proof}By definition if $L$ is $K_X$-negative, so is $L'$ for all $0<L'<L$.  We prove the proposition by induction.  If $\ls=\ls^{int}$, then $\ts^a(L,\chi)$ is empty and there is nothing to prove.  We assume $dim~\ts^{a'}(L',\chi')\leq L'^2-s_{L'}$ for all $0<L'<L$.  Then $dim~\mm^{a'}_{\bullet}(L',\chi')\leq \max\{L'^2,L'^2-s_{L'}\}$ for any $a'$ and $\chi'$.  

Let $F\in\ts^a(L,\chi)$, then $F$ is strictly semistable or unstable.  Hence we can have the following sequence
\begin{equation}0\ra F_1\ra F\ra F_2\ra0,
\end{equation}
with $F_i\in\mm_{\bullet}^{a_i}(L_i,\chi_i)$ for $i=1,2$ and $\text{Ext}^2(F_2,F_1)=0$.  Since  $\mu(F_2)=\frac{\chi_2}{L_2.\mo_X(1)}\leq\frac{\chi_1}{L_1.\mo_X(1)}=\mu(F_1)$, and $\chi_1\leq a$, there are finitely many possible choices for $((L_1,\chi_1),(L_2,\chi_2))$,  and we can also find upper bounds for $a_i$ (e.g. $a_1\leq a$ and $a_2\leq aL.\mo_X(1)$). 

Recall that $\chi(F_2,F_1):=\sum_{i}(-1)^idim~\text{Ext}^i(F_2,F_1)$.
The stack $\mext^1(F_2,F_1)$ has dimension $\leq \chi(F_2,F_1)$, because $\textbf{1}+\text{Hom}(F_2,F_1)$ is contained in the automorphism groups of all elements in $\text{Ext}^1(F_2,F_1)$ as in the following diagram.
\begin{equation}\label{exi}\xymatrix@C=2.5cm{0\ar[r] &F_1\ar[r]
\ar[d]_{Id} 
&F\ar[r]\ar[d]_{\cong}^{\varphi\in\textbf{1}+\text{Hom}(F_2,F_1)} &F_2\ar[r] \ar[d]^{Id}&0 \\
0\ar[r] &F_1\ar[r] &F\ar[r] &F_2\ar[r] &0.}\end{equation}

Hence $dim~\mext^1(F_2,F_1)\leq dim~\text{Ext}^1(F_2,F_1)-dim~\text{Hom}(F_2,F_1)=\chi(F_2,F_1)$ by $\text{Ext}^2(F_2,F_1)=0$. 

By induction assumption we have $dim~\mm_{\bullet}^{a_i}(L_i,\chi_i)\leq \max\{L_i^2,L_i^2-s_{L_i}\}$.  By Hirzebruch-Riemman-Roch, $\chi(F_2,F_1)=L_1.L_2=\frac12(L^2-L_1^2-L_2^2)$.  Hence we have 
\begin{eqnarray}dim~\ts^a(L,\chi)&\leq& \frac12(L^2-L_1^2-L_2^2) +\max\{L_1^2,L_1^2-s_{L_1}\}+\max\{L_2^2,L_2^2-s_{L_2}\}\nonumber\\
&\leq&\frac12(L^2+\max_{\footnotesize{\begin{array}{c}\{L_k\}_k;\\ \sum_k L_k=L;\\ \forall k, 0<L_k<L\end{array}}}\sum_k L^2_k)=L^2-s_L.\end{eqnarray}  
Hence the proposition.
\end{proof}

In this paper we mainly focus on the case $s_L>0$ and $dim~\mm^a_{\bullet}(L,\chi)=dim~\mm(L,\chi)=L^2$.  We have the following two useful lemmas.

\begin{lemma}\label{psc}If $\ls^{int}\neq\emptyset$ and $\chi(L)=h^0(L)$, then $s_L>0$.
\end{lemma}
\begin{proof}Since $X$ is rational, $H^2(\tilde{L})=0$ for any $\tilde{L}$ effective, hence $\chi(\tilde{L})\leq h^0(\tilde{L})$.  Because $\ls^{int}\neq\emptyset$ and $\chi(L)=h^0(L)=dim~\ls+1$, we have 
\begin{eqnarray}\label{get}0&<&dim~\ls-\max_{\footnotesize{\begin{array}{c}\{L_k\}_k;\\ \sum_k L_k=L;\\ \forall k, 0<L_k<L\end{array}}}\sum_k dim~|L_k|\nonumber\\ &\leq& (\chi(L)-1)-\max_{\footnotesize{\begin{array}{c}\{L_k\}_k;\\ \sum_k L_k=L;\\ \forall k, 0<L_k<L\end{array}}}\sum_k (\chi(L_k)-1)=s_L. 
\end{eqnarray}
The last equation is because $\chi(\tilde{L})-1=\frac12(-K_X.\tilde{L}+\tilde{L}^2)$.  Hence the lemma.
\end{proof}
\begin{lemma}\label{inch}If $\forall~0<L'\leq L$ and $|L'|^{int}\neq \emptyset$ we have $h^0(L')=\chi(L')$, then 
\[s_L=\min_{\footnotesize{\begin{array}{c}\{L_k\}_k; \\ \sum_k L_k=L;\\ \forall k, 0<L_k<L\\and~|L_k|^{int}\neq\emptyset\end{array}}}\displaystyle{\sum_{i<j}}L_i.L_j=\frac12(L^2-\max_{\footnotesize{\begin{array}{c}\{L_k\}_k; \\ \sum_k L_k=L;\\ \forall k, 0<L_k<L\\and~|L_k|^{int}\neq\emptyset\end{array}}}\sum_k L^2_k).\]
\end{lemma}
\begin{proof}If there is some $0<L_k<L$ such that $|L_k|^{int}=\emptyset$, then 
\begin{eqnarray}\label{down}0&=&dim~|L_k|-\max_{\footnotesize{\begin{array}{c}\{L^j_k\}_j;\\ \sum_j L^j_k=L_k;\\ \forall j, 0<L^j_k<L_k\\and~|L_k^j|^{int}\neq\emptyset\end{array}}}\sum_j dim~|L_k^j|\nonumber\\ &\geq& (\chi(L_k)-1)-\max_{\footnotesize{\begin{array}{c}\{L^j_k\}_j;\\ \sum_j L^j_k=L_k;\\ \forall j, 0<L^j_k<L_k\\and~|L_k^j|^{int}\neq\emptyset\end{array}}}\sum_j (\chi(L^j_k)-1)\nonumber\\&=& L^2_k-\max_{\footnotesize{\begin{array}{c}\{L^j_k\}_j;\\ \sum_j L^j_k=L_k;\\ \forall j, 0<L^j_k<L_k\\and~|L_k^j|^{int}\neq\emptyset\end{array}}}\sum_j (L^j_k)^2. 
\end{eqnarray}
Hence one can replace $L_k^2$ by $\sum_j (L^j_k)^2$ and this won't change $s_L$ and hence the lemma.
\end{proof}

\section{Irreducibility of $M^{ss}(L,\chi)$.}
In this section we will show the irreducibility of the moduli scheme $M^{ss}(L,\chi)$ under some suitable condition, which generalizes Theorem 3.1 in \cite{lee} and Corollary 4.2.9 in \cite{yuan}.
\begin{defn}Let $\mn(L,\chi)$ be the substack of $\mm_{\bullet}^a(L,\chi)$ parametrizing sheaves in $\mm^a_{\bullet}(L,\chi)$ with integral supports.  Let $N(L,\chi)$ be the image of $\mn(L,\chi)$ in the (coarse) moduli space $M(L,\chi)$.  
\end{defn} 

\begin{rem}It is obvious that $\mn(L,\chi)\subset\mm(L,\chi)$ and $\mn(L,\chi)$ does not depend on $a$ or the polarization. 
\end{rem}
\begin{lemma}\label{iid}If $\ls^{int}\neq\emptyset$ and $L.K_X<0$, then $N(L,\chi)$ is irreducible, smooth of dimension $L^2+1$.  Moreover, $h^0(L)=\chi(L)$.
\end{lemma}
\begin{proof}$N(L,\chi)$ is a family of (compactified) Jacobians over $\ls^{int}$, hence it is connected of dimension $dim~\ls+g_L=h^0(L)+\frac12(K_X.L+L^2)$.  

$L.K_X<0$ implies that $M(L,\chi)$ is smooth of dimension $L^2+1$ at every point $[F]\in N(L,\chi)$, hence $dim~N(L,\chi)\leq L^2+1=\chi(L)+\frac12(K_X.L+L^2)$.  On the other hand $h^0(L)\geq\chi(L)$ and hence $h^0(L)=\chi(L)$ and $N(L,\chi)$ is irreducible because it is smooth and connected. 
\end{proof}

We have a morphism $\pi:\mm_{\bullet}^a(L,\chi)\ra\ls$ sending every sheaf to its support.  Denote by $\ls^{R}$ the locally closed subscheme parametrizing sheaves with reducible supports, and $\ls^{N}$ the closed subscheme parametrizing sheaves with irreducible and non-reduced supports, i.e. of the form $kC$ with $k>1$ and $C\in\frac1{k}L|^{int}$.   We have that $\ls=\ls^{int}\cup\ls^R\cup\ls^N$ and $\ts^a(L,\chi)\subset\pi^{-1}(\ls^R\cup\ls^N)$.  

Let $\mc_R(d,\chi):=\pi^{-1}(\ls^R)\cap\mm(L,\chi)$ and $\mc_N(d,\chi):=\pi^{-1}(\ls^N)\cap\mm(L,\chi)$.   
\begin{lemma}\label{codr}If $L$ is $K_X$-negative, then $dim~\mc_R(d,\chi)\leq L^2-s_L$.
\end{lemma}
\begin{proof}We can use the same strategy as in Proposition \ref{dlt}.  Hence it is enough to show that every sheaf $F\in\mc_R(d,\chi)$ can be written as an extension of $F_2\in\mm^{a_2}_{\bullet}(L_2,\chi_2)$ by $F_1\in\mm^{a_1}_{\bullet}(L_1,\chi_1)$ with $\text{Ext}^2(F_2,F_1)=0$, and moreover there are finitely many possible choices of $((L_1,\chi_1),(L_2,\chi_2))$ and we can find upper bounds for $a_i$.

Let $C$ be the support of $F\in\mc_R(d,\chi)$.  $C$ is reducible, so we can write $C=C_1\cup C_2$ such that $C_1\cap C_2$ is 0-dimensional.  Let $L_i$ be the line bundle associated to the divisor class of $C_i$.  Then we have two exact sequences.
\begin{equation}\label{sfot}0\ra\mo_{C_1}(-L_2)\ra\mo_C\ra\mo_{C_2}\ra0;
\end{equation}
\begin{equation}\label{sfto}0\ra\mo_{C_2}(-L_1)\ra\mo_C\ra\mo_{C_1}\ra0.
\end{equation}
Tensor (\ref{sfot}) and (\ref{sfto}) by $F$ and we get
\begin{equation}\label{tsfot}Tor^1(F,\mo_{C_2})\xrightarrow{\jmath_1} F(-L_2)|_{C_1}\xrightarrow{\imath_1} F \ra F|_{C_2}\ra0;
\end{equation}
\begin{equation}\label{tsfto}Tor^1(F,\mo_{C_1})\xrightarrow{\jmath_2} F(-L_1)|_{C_2}\xrightarrow{\imath_2} F\ra F|_{C_1}\ra0.
\end{equation} 
Let $F_i^{tf}$ be the quotient sheaf of $F|_{C_i}$ module its maximal 0-dimensional subsheaf.  Then the image of $\imath_1$ is $F^{tf}_1(-L_2)$, because the image of $\jmath_1$ is supported at $C_1\cap C_2$ and hence a 0-dimensional subsheaf in $F(-L_2)|_{C_1}$ and $F$ is pure.  The same holds for $\imath_2$.  Hence we have
\begin{equation}\label{psfot}0\ra F_1^{tf}(-L_2)\ra F \xrightarrow{p_2} F|_{C_2}\ra0;
\end{equation}
\begin{equation}\label{psfto}0\ra F_2^{tf}(-L_1) \ra F\ra F|_{C_1}\ra0.
\end{equation} 
Compose map $p_2$ with the surjection $F|_{C_2}\ra F_2^{tf}$, and we get a sequence as follows.
\begin{equation}\label{fsfot}0\ra F_1\ra F \ra F_2^{tf}\ra0;
\end{equation}
where $F_1$ is the extension of the maximal 0-dimensional subsheaf of $F|_{C_2}$ by $F_1^{tf}(-L_2)$.  Hence $a\geq\chi(F_1)\geq\chi(F_1^{tf}(-L_2))=\chi(F_1^{tf})-L_2L_1$.   The same holds for $F^{tf}_2$ and hence we have $\chi(F_2^{tf})\leq a+L_1.L_2$.  Moreover for every subsheaf $G\subset F_2^{tf}$, by (\ref{psfto}) $G(-L_1)$ is a subsheaf of $F$, hence $\chi(G(-L_1))=\chi(G)-c_1(G).L_1\leq a$, and hence $\chi(G)\leq a+c_1(G).L_1$. 

Now (\ref{fsfot}) gives us the extension we need: $F_1\in\mm^a_{\bullet}(L_1,\chi_1)$, $F_2^{tf}\in\mm^{a+L_1.L_2}_{\bullet}(L_2,\chi_2)$;  and since $C_1\cap C_2$ is of 0-dimensional and both $F_1$ and $F_2^{tf}$ are pure of dimensional 1, $\text{Hom}(F_1,F_2^{tf}(K_X))=0$ and hence $\text{Ext}^2(F_2^{tf},F_1)=0$.  For fixed $(L,\chi,a)$, there are finitely many possible choices of $((L_1,\chi_1),(L_2,\chi_2))$.  Hence the lemma.
\end{proof}

The dimension of $\mc_N(L,\chi)$ is more complicated to estimate and the result is not so neat as $\mc_R(L,\chi)$.  We will do it in Section 4.  At this moment we can conclude the following theorem.
\begin{thm}\label{pride}Let $L$ be $K_X$-negative with $\ls^{int}\neq\emptyset$, and moreover let $L$ be primitive, i.e. $L\neq nL'$ for any $n\in\mathbb{Z}_{>1}$ and $L'\in Pic(X)$.  Then $M^{ss}(L,\chi)$ is irreducible of dimension $L^2+1$.
\end{thm}
\begin{proof}By Lemma \ref{psc} and Lemma \ref{iid}, $s_L>0$.  The stack $\mm^{ss}(L,\chi)$ has an atlas $\Omega^{ss}(L,\chi)$, which is an open subset of some Quot-scheme, such that the morphism $\phi:\Omega^{ss}(L,\chi)\ra M^{ss}(L,\chi)$ is a good quotient.  It is enough to show that $\Omega^{ss}(L,\chi)$ is irreducible.  

Since $L$ is $K_X$-negative, $\Omega^{ss}(L,\chi)$ can be chosen to be smooth, hence it is irreducible if it is connected.  Since $L$ is primitive, $\ls^N=\emptyset=\mc_N(L,\chi)$.  The connectedness of $\Omega^{ss}(L,\chi)$ follows immediately from Lemma \ref{iid}, Lemma \ref{codr} and the fact that $\Omega^{ss}(L,\chi)$ is an atlas of the stack $\mm^{ss}(L,\chi)$.  Hence the theorem.  
\end{proof}
 
\begin{example}\label{exl}Let $X=\p(\mo_{\p^1}\oplus\mo_{\p^1}(e))$ with $e=0,1$.  Denote by $f$ and $\sigma$ the fiber class and section class such that $\sigma^2=-e$.  Then Theorem \ref{pride} applies to $L=a\sigma+bf$ such that $a>0$, $b>ae$ and $g.c.d(a,b)=1$.  In this case $s_L=\min\{e+(b-ae),a\}$.
\end{example}
\section{Sheaves with non-reduced supports.}  
Let $\mc_{\frac Lk}\subset\mc_N(d,\chi)$ be the substack parametrizing sheaves with supports $kC$ for $C\in|\frac Lk|^{int}$.  Hence $\mc_N(d,\chi)$ is a disjoint union of $\mc_{\frac Lk}$ with $k\in\mathbb{Z}_{>1}$ and $\frac Lk\in Pic(X)$. 

In this section, we ask $L^2\geq0$.  This because if $L^2<0$, then $L^2<-1$ since $L$ is not primitive.  Then $M(L,\chi)$ must be empty and there is nothing to worry about.

Recall that we have defined a stack $\mm^a_{\bullet}(L,\chi)$.  Let $\mc_{\frac Lk,a}$ be the substack of $\mm^a_{\bullet}(L,\chi)$ parametrizing sheaves with support $kC$ for some $C\in|\frac Lk|^{int}$.  Hence $\mc_{\frac Lk}\subset\mc_{\frac Lk,a}$.
\begin{flushleft}{\textbf{$\dagger$ \large{$\mc_{\frac Lk}$ for $g_{\frac Lk}=0$.}}}\end{flushleft} 

\begin{prop}\label{ckot}Let $L.K_X<0$, $L^2\geq0$ and let $g_{\frac Lk}=0$, then $dim~\mc_{\frac Lk}\leq dim~\mc_{\frac Lk,a}\leq L^2-\frac {k-1}k L^2\leq L^2-s_L$.
\end{prop}
\begin{proof}We use the same strategy again as in Proposition \ref{dlt} and Lemma \ref{codr}, and the proposition follows immediately from the following lemma.
\end{proof}
\begin{lemma}\label{po}Let $F$ be a pure sheaf with support $k C$ on any surface $X$, such that $C\cong\p^1$.  Let $\xi=C.C$ be the self intersection number of $C$.  Assume moreover $\xi\geq0$.  Then $F$ admits a filtration
\[0=F_0\subsetneq F_1\subsetneq\cdots\subsetneq F_r=F,\]
such that $F_{i}/F_{i-1}\cong\mo_{\p^1}(s_i)$ and $s_i-s_{i+1}\geq -\xi$.  Moreover we can ask such filtration also to satisfy that 
\[\forall ~0<i\leq r, \text{Ext}^2(F/F_i,F_i)^{\vee}\cong \text{Hom}~(F_i,F/F_i(K_X))=0.\]   
\end{lemma}
\begin{proof}Since $C\cong\p^1$, every pure sheaf on $C$ is locally free and splits into the direct sum of line bundles.  Now take an exact sequence on $X$
\[0\ra\mo_{C}(s_1)\ra E\ra\mo_C(s_2)\ra0.\] 
We claim that if $s_1<s_2-\xi$, then $E$ is a locally free sheaf of rank 2 on $C$ and hence $E$ splits into direct sum of two line bundles.   

Denote by $\text{Ext}^1_{C}(\mo_C(s_2),\mo_C(s_1))$ the group of extensions of $\mo_C(s_2)$ by $\mo_C(s_1)$ as sheaves of $\mo_C$-modules.  Each sheaf in $\text{Ext}^1_{C}(\mo_C(s_2),\mo_C(s_1))$ is a rank 2 bundle on $C$.  Notice that $\text{Ext}^1_{C}(\mo_C(s_2),\mo_C(s_1))$ is a linear subspace inside $\text{Ext}^1(\mo_C(s_2),\mo_C(s_1))$, since every non-split extension in $\text{Ext}^1_{C}(\mo_C(s_2),\mo_C(s_1))$ is a non-split extension in $\text{Ext}^1(\mo_C(s_2),\mo_C(s_1))$.  So to prove the claim, we only need to show the following statement.
\begin{equation}\label{exd}dim~\text{Ext}^1_{C}(\mo_C(s_2),\mo_C(s_1))=dim~\text{Ext}^1(\mo_C(s_2),\mo_C(s_1)),\forall ~s_1<s_2-\xi.\end{equation}
The LHS is easy to compute and we get LHS$=dim~H^1(\mo_{\p^1}(s_1-s_2))=s_2-s_1-1$.  Since $\xi\geq0$ and $s_1< s_2-\xi$, $s_2-s_1-1\geq0$.  

$\chi(\mo_C(s_2),\mo_C(s_1))=-C.C=-\xi$ by Hirzebruch-Riemman-Roch on $X$.  

$\text{Hom}(\mo_C(s_2),\mo_C(s_1))=0$ since $s_1< s_2$.  $dim~\text{Ext}^2(\mo_C(s_2),\mo_C(s_1))=$ $dim~\text{Hom}(\mo_C(s_1),\mo_C(s_2+K_X))$ by Serre duality.  The canonical line bundle on $C$ is given by $K_X\otimes\mo_X(C)|_C$ and isomorphic to $\mo_{\p^1}(-2)$, hence $K_X.C+C.C=-2$ and hence $K_X.C=-2-\xi.$  Therefore, $dim~\text{Hom}(\mo_C(s_1),\mo_C(s_2+K_X))=s_2-s_1-\xi-1\geq 0.$  Finally we have $dim~\text{Ext}^1(\mo_C(s_2),\mo_C(s_1))=s_2-s_1-1$.  Hence (\ref{exd}) holds.

Now we construct a filtration as follows.  We choose $F_1\cong\mo_C(s_1)$ to be the subsheaf supported on $C$ with rank 1 and the maximal degree, i.e. $\forall F'_1\subset F, F'_1\cong\mo_C(s'_1)$, then we have $s'_1\leq s_1$.  Apply induction assumption to $F/F_1$ and we then get a filtration.  It is easy to check that this filtration satisfies the property in the lemma.  Hence we proved the lemma.  
\end{proof}
\begin{rem}(1) Proposition 3.4 in \cite{moz} is a special case for Lemma \ref{po} with $\xi=0$.

(2) For sheaves $F_1$ and $F_2$ supported at an integral curve $C$, $\text{Ext}_C^i(F_1,F_2)$ is in general not a subspace of $\text{Ext}^i(F_1,F_2)$ for $i\geq 2$, i.e. the map $\text{Ext}_C^i(F_1,F_2)\ra\text{Ext}^i(F_1,F_2)$ might not be injective.
\end{rem}
\begin{flushleft}{\textbf{$\dagger$\large{ $\mc_{\frac Lk}$ for $g_{\frac Lk}>0$ and $k=2$.}}}\end{flushleft} 

\begin{prop}\label{coha}If $L.K_X<0$ and $g_{\frac Lk}>0$, then $dim~\mc_{\frac L2}\leq dim~\mc_{\frac L2,a}\leq L^2+L.K_X+1+(1-g_{\frac L2})\leq L^2+L.K_X+1$.  In particular if $L+K_X>0$, $-K_X>0$ and $K_X^2\geq 1$, then $dim~\mc_{\frac L2,a}\leq L^2-s_L$.
\end{prop}
\begin{proof}$g_{\frac L2}>0\Rightarrow L^2>0$.  According to the stratification (\ref{stmc}), $\mc_{\frac L2,a}$ only has two strata: $\mc_{\frac L2,a}^{1,1}$ and $\mc_{\frac L2,a}^{2}$.  We know that $dim~\mc_{\frac L2,a}^{1,1}\leq L^2+K_X.L+1+(1-g_{\frac L2})$ by (\ref{efft}) in the proof of Lemma \ref{rko}.  Hence we only need to estimate $dim~\mc_{\frac L2,a}^{2}$.  

Sheaves in $\mc_{\frac L2,a}^{2}$ are rank 2 torsion free sheaves on some integral curve $C$ in $|\frac L2|$.  Let $F\in\mc_{\frac L2,a}$.  By replacing $F$ by $F(nK_X)$ or $\mathcal{E}xt^1(F,mK_X)$ for some suitable $n$ and $m$, we can assume $0<\chi\leq -\frac{K_X.L}2$.  Hence for every sheaf $F$ in $\mc_{\frac L2,a}^2$ with support $C$, there is a nonzero global section which has to be a injection since both $\mo_C$ and $F$ are pure and $C$ is integral.  Hence we have the following sequence.
\begin{equation}\label{had}0\ra\mo_C\ra F\ra \widehat{I}\ra 0.
\end{equation}
The quotient $\widehat{I}$ may not be torsion free on $C$.  Take $I_2$ to be the quotient of $\widehat{I}$ module its torsion.  Then we have another exact sequence as follows.
\begin{equation}\label{half}0\ra I_1\ra F\ra I_2\ra0,
\end{equation}
where $I_1$ is a torsion free rank 1 sheaf with non-negative degree.  Let $\chi_i=\chi(I_i)$.  We have $\chi(\mo_C)\leq\chi_1\leq \max\{\chi,a\}$, hence there are finitely many possible choices for $(\chi_1,\chi_2)$.  Notice that $(\ref{half})$ gives an element in $\text{Ext}^1_C(I_2,I_1)$ which is a linear subspace inside $\text{Ext}^1(I_2,I_1)$.

If there is a number $N$ satisfying that $dim~\text{Ext}^2(I_2,I_1)\leq N$ for all $I_i$ in (\ref{half}) with $F\in \mc^{2}_{\frac L2,a}$, then using analogous argument to Proposition \ref{dlt} we can easily deduce the following estimate.
\begin{equation}\label{dha}dim~\mc_{\frac L2,a}^{2}\leq dim~|\frac L2|+2g_C-\chi(I_2,I_1)+N-2,
\end{equation}
We can find a suitable $N$ to bound $dim~\text{Ext}^2(I_2,I_1)$ as follows.
\begin{eqnarray}&dim&\text{Ext}^2(I_2,I_1)= dim~\text{Hom}(I_1,I_2(K_X))\nonumber\\
&\leq &dim~\text{Hom}(\mo_C,I_2(K_X))= h^0(I_2(K_X))\leq deg(I_2(K_X))+1\nonumber\\
&\leq& deg(\widehat{I}(K_X))+1=\frac{K_X.L}2+\chi+2g_C-1.
\end{eqnarray}  
Let $N=\frac {K_X.L}2+\chi+2g_C-1$.  By Lemma \ref{iid}, $dim~|\frac L2|=\frac12(\frac{L^2}4-\frac{K_X.L}2)$.  Hence (\ref{dha}) gives the following equation.
\begin{equation}\label{dhal}dim~\mc_{\frac L2,a}^{2}\leq \frac{L^2}2+3g_C-2+\frac{K_X.L}2+\chi\leq L^2+K_X.L+1+(1-g_{\frac L2}).
\end{equation}
The last equation is because  $\chi\leq -\frac{K_X.L}2$.  Hence the proposition.
\end{proof}

Notice that since $L$ is not primitive, $K_X.L<0$ implies that $K_X.L<-1$ hence $K_X.L+1<0$.  Lemma \ref{codr}, Proposition \ref{ckot} and  Proposition \ref{coha} together give the following theorem.
\begin{thm}\label{prime}Let $L$ be $K_X$-negative such that $\ls^{int}\neq\emptyset$ and $L^2\geq0$, and moreover $L=nL'$ with $n\in\mathbb{Z}_{>1}$ and $L'$ primitive.  Then $M^{ss}(L,\chi)$ is irreducible if one of the following three conditions is satisfied.

(1) $n=2$; 

(2) $n$ is prime and either $|L'|^{int}=\emptyset$ or $g_{L'}=0$;

(3) $n=2p$ with $p$ prime and both $L'$ and $2L'$ satisfy (2).
\end{thm}
\begin{example}Theorem \ref{prime} applies to the following examples.

(1) $X=\p^2$, and $L=pH$ or $2pH$ with $H$ the hyperplan class;

(2) $X=\p(\mo_{\p^1}\oplus\mo_{\p^1}(e))$ with $e=0,1$,  and $L=a\sigma+bf$ such that $a>0$, $b>ae$ and $g.c.d(a,b)=2$, or $L=p(\sigma+cf)$ with $c>e$ and $p$ prime, where $\sigma$ and $f$ are the same as in Example \ref{exl}.
\end{example}
\begin{flushleft}{\textbf{$\dagger$\large{ $\mc_{\frac Lk}$ in general.}}}\end{flushleft} 
\begin{prop}\label{gck}Let $F\in\mc_{\frac Lk,a}$ with support $kC$ and $C\in|\frac Lk|^{int}$, then there is a filtration of $F$ 
\[0=F_0\subsetneq F_1\subsetneq\cdots\subsetneq F_l=F,\]
such that $Q_i:=F_{i}/F_{i-1}$ are torsion-free sheaves on $C$ with rank $r_i$.  $\sum r_i=k$, and moreover there are injections $f^i_F:Q_{i}(-C)\hookrightarrow Q_{i-1}$ induced by $F$ for all $2\leq i\leq l$.
\end{prop}
\begin{proof}Let $\delta_C$ be the function defining the curve $C$.   Since $C$ is integral, $\delta_C$ is irreducible.  For a sheaf $F\in\mc_{\frac Lk,a}$ with reduced support $C$,  $\exists~ l\in\mathbb{Z}_{>0}$ such that $\delta_C^l\cdot F=0$ and $\delta_C^{l-1}\cdot F\neq0.$  Take $F_1$ to be the subsheaf of all the annihilators of $\delta_C$, i.e. $F_1(U):=\{e\in F(U)|\delta_C\cdot e=0\},\forall~ U$ open.  $F_1$ is a pure 1-dimensional sheaf of $\mo_C$-module and hence it is a torsion free sheaf on $C$.  $F/F_1$ is pure of dimension 1, because $F_1$ is the maximal subsheaf of $F$ supported on $C$.
Apply the induction assumption to $F/F_1$,  and we get a filtration $0=F_0\subsetneq F_1\subsetneq\cdots\subsetneq F_l=F$ with $Q_i:=F_{i}/F_{i-1}$ torsion-free on $C$.

We want to show there are injective maps $f_F^i:Q_i(-C)\hookrightarrow Q_{i-1}$.  By induction, it is enough to construct the map $f_F^2:Q_2(-C)\hookrightarrow Q_1$.  We have the following exact sequence.
\begin{equation}\label{ftwo}0\ra Q_1\ra F_2\ra Q_2\ra 0.
\end{equation}
By the definition, we know that $\delta_C\cdot F_2\neq 0$ and $\delta^2_C\cdot F_2=0$.  Hence multiplying $\delta_C$ gives a non-zero map $m_C:F_2(-C)\ra F_2$ with the kernel $Q_1(-C)$ and the image contained in $Q_1$.  Hence $m_C$ induces an injective map $f_F^2:Q_2(-C)\hookrightarrow Q_1$.   Hence the proposition.
\end{proof}

Propositon \ref{gck} implies that we have a morphism from $\mc_{\frac Lk,a}$ to some Flag scheme by sending $F$ to $(Q_l\subset Q_{l-1}(C)\subset\cdots\subset Q_1((l-1)C))$.  But still it is difficult to compute its dimension in general.
\begin{rem}The filtration constructed in the proof of Proposition \ref{gck} is unique.   Hence we stratify $\mc_{\frac Lk,a}$ ($\mc_{\frac Lk}$, resp.) by the ranks $r_i$ of the factors $Q_i$ as follows.
 \begin{equation}\label{stmc}\mc_{\frac Lk,a} ~(\mc_{\frac Lk},~resp.)=\displaystyle{\coprod_{\begin{array}{c}r_1\geq \cdots\geq r_l>0,\\ \sum r_i=k.\end{array}}} \mc_{\frac Lk,a}^{r_1,\cdots,r_l}~(\mc_{\frac Lk}^{r_1,\cdots,r_l},~resp.).\end{equation}
\end{rem}

\begin{lemma}\label{rko}Let $L.K_X<0$ and $g_{\frac Lk}>0$.  Then $dim~\mc_{\frac Lk}^{1,1,\cdots,1}\leq dim~\mc_{\frac Lk,a}^{1,1,\cdots,1}\leq L^2+K_X.L+1+(2k-3)(1-g_{\frac Lk})\leq L^2+K_X.L+1$.
\end{lemma}
\begin{proof}In this case we have $l=k\geq2$.  It is easy to check for given $(L,\chi)$ there are finitely many possible choices for $(c_1(Q_i),\chi(Q_i))$, where $Q_i$ are the factors in the filtration in Proposition \ref{gck}.  Actually we have $c_1(Q_i)=\frac Lk$, $\chi(Q_i)\geq\chi(Q_{i+1})-(\frac Lk)^2$, $\displaystyle{\sum_{i=1}^s}\chi(Q_i)\leq\max\{a, \frac sk\chi\}$ for all $s<k$ and finally $\displaystyle{\sum_{i=t}^k}\chi(Q_i)\geq \min\{\chi-a,\frac{k-t+1}k\chi\}$ for all $1<t\leq k$.  By the finiteness of $\{(c_1(Q_i),\chi(Q_i))\}$, we can estimate the dimension of $\mc^{1,\cdots,1}_{\frac Lk,a}$ for some fixed $(c_1(Q_i)=\frac Lk,\chi(Q_i))$. 

We first prove the lemma for $l=2$.  Let $F\in\mc_{\frac L2,a}^{1,1}$.  Then $F$ can be fit in the following sequence.
\begin{equation}\label{let}0\ra Q_1\ra F\ra Q_2\ra 0.
\end{equation}
Let $C$ be the reduced support of $F$.  By Proposition \ref{gck} we have $Q_i$ are torsion free of rank 1 on $C$ and there is an injection $f:Q_2(-C)\hookrightarrow Q_1$.  The parametrizing space of rank 1 torsion free sheaves on $C$ is its compactified Jacobian and well-known to be integral with dimension the arithmetic genus $g_C$ of $C$ (see \cite{aik}).  If there is a number $N$ satisfying that $dim~\text{Ext}^2(Q_2,Q_1)\leq N$ for all $Q_i$ in (\ref{let}) with $F\in \mc^{1,1}_{\frac L2,a}$, then using analogous argument to Proposition \ref{dlt} we can easily deduce the following estimate.
\begin{equation}\label{eft}dim~\mc_{\frac L2,a}^{1,1}\leq dim~|\frac{L}2|+g_C+g_C-\chi(Q_2,Q_1)+N-2.
\end{equation}
$g_C=g_{\frac L2}=\frac{K_X.L}4+\frac{L^2}8+1,$ and $\chi(Q_2,Q_1)=-C.C=-\frac{L^2}4$. 

We need a upper bound $N$ of $dim~\text{Ext}^2(Q_2,Q_1)=dim~\text{Hom}(Q_1,Q_2(K_X))$.  Since there is an injection from $Q_2(-C)$ to $Q_1$ with cokernel 0-dimensional, $\text{Hom}(Q_1,Q_2(K_X))$ is a subspace of $\text{Hom}(Q_2(-C),Q_2(K_X))$.  Since $C$ is Gorenstein with dualizing sheaf $\omega_C$ and $\mo_C(K_X+C)\cong\omega_C$, we have
\begin{eqnarray}&dim&\text{Ext}^2(Q_2,Q_1)= dim~\text{Hom}(Q_1,Q_2(K_X))\nonumber\\
&\leq &dim~\text{Hom}(Q_2(-C),Q_2(K_X))\nonumber\\&=& dim~\text{Hom}(Q_2,Q_2\otimes\omega_C) 
\leq deg(\omega_C)+1=2g_C-1.
\end{eqnarray}  
Let $N=2g_C-1$ and by Lemma \ref{iid}, $\chi(\frac L2)=h^0(\frac L2)$.  Hence (\ref{eft}) gives the following equation.
\begin{equation}\label{efft}dim~\mc_{\frac L2,a}^{1,1}\leq \frac {L^2}2+3g_C-2=L^2-(g_C-1)+K_X.L+1\leq L^2+K_X.L+1.
\end{equation}
Hence we proved the lemma for $l=2$.

Let $l\geq 3$.  Let $F\in\mc^{1,\cdots,1}_{\frac Ll,a}$ and take the filtration of $F$ as given in Proposition \ref{gck}.  Then we have the following sequence.
\begin{equation}\label{esg}0\ra F_1\ra F\ra F/F_1\ra 0.
\end{equation}
If $\exists~ N$ such that $dim~\text{Hom}(F_1,F/F_1(K_X))\leq N$ for all $F_1$ in (\ref{esg}) with $F\in \mc^{1,\cdots,1}_{\frac Ll,a}$, then by induction assumption we have the following estimate.
\begin{eqnarray}\label{estg}&dim&\mc_{\frac Ll,a}^{1,\cdots,1}\leq dim~\mc^{1,\cdots,1}_{\frac{(l-1)L}{(l-1)l},a'}+g_C-1-\chi(F/F_1,F_1)+N\nonumber\\
&\leq & (\frac{l-1}l)^2 L^2+(\frac{l-1}lK_X.L+1)+g_C-1+\frac{l-1}{l^2}L^2+N
\end{eqnarray}

Notice that any nonzero map $F_1\ra F/F_1(K_X)$ has its image annihilated by $\delta_C$ and hence contained in $Q_2(K_X)=F_2/F_1(K_X)$.  Thus $\text{Hom}(F_1,F/F_1(K_X))=\text{Hom}(F_1,Q_2(K_X))$ and then by the same argument as we did for $l=2$, we can let $N$ in (\ref{estg}) to be $2g_C-1$.   Therefore
\begin{eqnarray}\label{esff}&&dim~\mc_{\frac Ll,a}^{1,\cdots,1}\leq L^2-\frac {L^2}{l}+3g_C-2+(\frac{l-1}lK_X.L+1)\nonumber\\
&=& L^2+(2l-3)(1-g_C)+(K_X.L+1)+(\frac{l-1}lK_X.L+1)\nonumber\\&\leq& L^2+(K_X.L+1)+(2l-3)(1-g_{\frac Ll})\leq  L^2+(K_X.L+1).
\end{eqnarray}
The second equality is because $g_C-1=\frac{K_X.L}{2l}+\frac{L^2}{2l^2}$.  Hence the lemma.
\end{proof}

For general $(r_1,\cdots,r_l)$, at the moment we still don't have an estimate for $dim~\mc_{\frac Lk}^{r_1,\cdots,r_l}$ as good as Lemma \ref{rko}.  However for some special $X$, such as $\p^2$ and $\p(\mo_{\p^1}\oplus\mo_{\p^1}(e))$ with $e=0,1$, we have a weaker result.

We first need to introduce more properties of sheaves with non-reduced supports. 

\begin{prop}\label{bcd}Let $F\in\mc_{\frac Lk,a}$ and let $C$ be the reduced curve in $Supp(F)$, then there is a filtration of $F$ 
\[0=F^0\subsetneq F^1\subsetneq\cdots\subsetneq F^m=F,\]
such that $R_i:=F^{i}/F^{i-1}$ are sheaves on $C$ with rank $t_i$.  $\sum t_i=k$, and moreover there are surjections $g^i_F:R_{i}(-C)\twoheadrightarrow R_{i-1}$ induced by $F$ for all $2\leq i\leq m$.  $R_i$ are not necessarily torsion free on $C$.
\end{prop}
\begin{proof}We choose $F^{m-1}$ to be the kernel of the map $F\twoheadrightarrow F\otimes\mo_C$, and hence $R_m\cong F\otimes\mo_C$.  $F^{m-1}$ is the quotient of $F\otimes\mo_{X}(-C)$ module $Tor^1_{\mo_X}(F,\mo_C)$,  hence we have a surjective map $g^m_F:R_m(-C)\twoheadrightarrow R_{m-1}:=F^{m-1}\otimes \mo_C$. We then get the proposition by induction.
\end{proof}

Compare the two filtrations in Proposition \ref{gck} and Proposition \ref{bcd}, then we have the following lemma.
\begin{lemma}\label{cop}Let $F\in\mc_{\frac Lk,a}$ and let $C$ be the reduced curve in $Supp(F)$.  Let $(l,r_i,Q_i)$ and $(m,t_i,R_i)$ be as in Proposition \ref{gck} and Proposition \ref{bcd} respectively.  Then we have 

(1) $l=m$;

(2) $\forall~1\leq i\leq m,~r_i=t_{m-i+1}$;

(3) $\forall~1\leq i\leq m,~\chi(R_{i})=\chi(Q_{m-i+1})+\sum_{j=1}^{i-1}r_jC^2.$
\end{lemma}
\begin{proof}Statement (1) is trivial, since both $m$ and $l$ are the minimal power of $\delta_C$ to annihilate $F$.

We first prove Statement (2) for $l=2$.  We denote by $\Pi_1$ the image of $f^2_F$ inside $F_1$, and $F/\Pi_1\cong F\otimes \mo_{(l-1)C}$.  Hence for $l=2$ $F/\Pi_1\cong F\otimes\mo_C\cong R_2$.  Hence $t_2=r_2+r_1-r_2=r_1$ and $t_1=r_2$.

Let $l\geq 3$.  Take the torsion free quotient $\widetilde{F}$ of $F/\Pi_1$ and we have $\widetilde{r}_1=r_2+r_1-r_2=r_1$, $\widetilde{r}_i=r_{i+1}$ for $i>1$, and $\widetilde{t}_{m-i}=t_{m-i+1}$ for $i\geq 1$.  Hence by induction assumption, we have $r_1=t_m$, $r_{i+1}=\widetilde{r}_{i}=\widetilde{t}_{m-1-i+1}=t_{m-i+1}$ for $i\geq 2$.  We then have $r_2=t_{m-1}$ because $\sum r_i=\sum t_i$.  Hence Statement (2).

We have the following exact sequence
\[0\ra Tor^1_{\mo_X}(F,\mo_C)\ra F(-C)\xrightarrow{\cdot\delta_C} F\ra F\otimes\mo_C\ra0.\]
By definition $R_m\cong F\otimes\mo_C$ and $Q_1(-C)= ker(\cdot\delta_C)\cong Tor^1_{\mo_X}(F,\mo_C)$.  Therefore, $\chi(R_m)=\chi(Q_1)-l_1C^2+kC^2=\chi(Q_1)+\sum_{j=1}^{m-1}r_jC^2$.  Notice that $l_1=r_m$ by Statement (2).  Then by applying the same argument to $F^{m-1}\cong F(-C)/Q_1(-C)$ we proved Statement (3).  
\end{proof}
\begin{defn}\label{uplow}We call the filtration in Proposition \ref{gck} \textbf{the lower filtration of $F$} while the one in Proposition \ref{bcd} \textbf{the upper filtration of $F$}.
\end{defn}
Define $\mm(L,\chi)\supset \mt_n(L,\chi):=\{F~|~\exists ~x\in X, s.t.~dim_{k(x)}(F\otimes k(x))\geq n\},$ where $k(x)$ is the residue field of $x$.  In other words, $\mt_n(L,\chi)$ is the substack parametrizing sheaves with fiber dimension $\geq n$ at some point. 
\begin{rem}\label{fib}For a stable sheaf $F$ with filtrations in Proposition \ref{gck} and Proposition \ref{bcd}, we have $F\in\mt_{n_0}(L,\chi)$ with $n_0=r_1=t_m$.
\end{rem}
\begin{prop}\label{ttc}If there is an ample class $\mathscr{H}$ such that $\forall ~0<L'\leq L$, $(\mathscr{H}+K_X).L'\leq0$ and $(\mathscr{H}+K_X).L<0$, then for $n\geq 2$, $\mt_n(L,\chi)$ is of codimension $\geq n^2-2$ in $\mm(L,\chi)$.
\end{prop}
\begin{proof}Recall that we have a coarse moduli space $M(L,\chi)$ as a scheme.  We denote $T_n(L,\chi)$ the image of $\mt_n(L,\chi)$ in $M(L,\chi)$.  This proposition is equivalent to say that $T_n(L,\chi)$ is of codimension $\geq n^2-2$ in $M(L,\chi)$.

We know that there is a Qout-scheme $\Omega(L,\chi)$ such that $\phi:\Omega(L,\chi)\ra M(L,\chi)$ is a $PGL(V)$-bundle.  By Le Potier's argument in the proof of Lemma 3.2 in \cite{lee}, the preimage $\phi^{-1}(T_n(L,\chi))$ of $T_n(L,\chi)$ is a closed subscheme of codimension $\geq n^2-2$ in $\Omega(L,\chi)$.  It is easy to see that $\phi^{-1}(T_n(L,\chi))$ is invariant under the $PGL(V)$-action, hence the proposition.    
\end{proof}
\begin{example}Proposition \ref{ttc} applies to the following examples.

(1) $X=\p^2$, and $L=dH$ with $H$ the hyperplan class;

(2) $X=\p(\mo_{\p^1}\oplus\mo_{\p^1}(e))$ with $e=0,1$,  and $L=a\sigma+bf$ such that $a\geq0$ and $b>0$, where $\sigma$ and $f$ are the same as in Example \ref{exl}.
\end{example}
If Proposition \ref{ttc} applies to $L$, then $\mt_3(L,\chi)$ is of dimension $\leq L^2-7$.  

Let $\mt_n^o(L,\chi)=\mt_n(L,\chi)-\mt_{n+1}(L,\chi)$.   
\begin{thm}\label{tt}Let $L.K_X<0$ and $g_{\frac Lk}>0$.  Then $dim~\mt_{2}^o(L,\chi)\cap \mc_{\frac Lk}\leq L^2+(K_X.L+1)$.
\end{thm}
\begin{proof}The proof is too long and moved to the appendix.
\end{proof}
Finally we state the following theorem which will play a key role in the next section where we compute several (virtual) Betti numbers of $M^{ss}(L,\chi)$ by computing its motivic measure.
\begin{thm}\label{entire}Let $L$ be $K_X$-negative such that $\ls^{int}\neq \emptyset$ and $L^2\geq0$.  Write $L=nL'$ with $L'$ primitive.  Assume moreover $L$ satisfies one of the following 4 conditions.

(1) $n=1$ or $2$; 

(2) $n$ is prime and either $|L'|^{int}=\emptyset$ or $g_{L'}=0$;

(3) $n=2p$ with $p$ prime and both $L'$ and $2L'$ satisfy (2);

(4) $\exists~ \mathscr{H}$ an ample class, such that $\forall ~0<L'\leq L$, $(\mathscr{H}+K_X).L'\leq0$ and $(\mathscr{H}+K_X).L<0$. 

Then there is a positive integer $\rho_L$ such that $\mm^{a}_{\bullet}(L,\chi)-\mn(L,\chi)$ is of codimension $\geq \rho_L$ in $\mm^{a}_{\bullet}(L,\chi)$ for any $a$ and $\chi$.  In particular $dim~\mm^{a}_{\bullet}(L,\chi)=dim~\mn(L,\chi)=L^2,~\forall~a,\chi$.
\end{thm}

\begin{example}(1) $X=\p^2$, and $L=dH$ with $H$ the hyperplan class.  Then $\rho_d=\rho_{dH}$ can be chosen as follows.
\begin{equation}\label{rhop}\rho_d:=\left\{\begin{array}{l}d-1,~for~d=p~or~2p~with~p~prime.\\ \\ 7,~otherwise.\end{array}\right.\end{equation}
(2) $X=\p(\mo_{\p^1}\oplus\mo_{\p^1}(e))$ with $e=0,1$,  and $L=a\sigma+bf$ such that $a>0$ and $b>ae$.  Then $\rho_L$ can be chosen as follows.
\begin{equation}\label{rhoh}\rho_L:=\left\{\begin{array}{l}\min\{b-(a-1)e,a\},~for~a~prime~or~g.c.d(a,b)=1~or~2;\\ \\ \min\{7,b-(a-1)e,a\},~otherwise.\end{array}\right.\end{equation}
\end{example}

\section{Motivic measures and the main theorem.}
In this section we will compute motivic measures of $\mm(L,\chi)$, or more precisely $\mn(L,\chi)$.  As we said in the introduction, our strategy is to relate the moduli stack $\mn(L,\chi)$ to the stack corresponding to some Hilbert scheme of points on $X$.  Let $\mh^n$ be the stack associated to the Hilbert scheme $Hilb^{[n]}(X)$ parametrizing ideal sheaves of colength $n$ on $X$.  Then $dim~\mh^{n}=2n-1.$  

We ask $L$ to be $K_X$-negative and let $\ls^{int}\neq\emptyset$, hence $\chi(L)=h^0(L)$

\begin{defn}For two integers $k>0$ and $i$, we define $\mm_{k,i}^a(L,\chi)$ to be the (locally closed) substack of $\mm_{\bullet}^a(L,\chi)$ parametrizing sheaves $F\in\mm_{\bullet}^a(L,\chi)$ with $h^1(F(-iK_X))=k$ and $h^1(F(-nK_X))=0,\forall n>i.$  

Let $\mn_{k,i}(L,\chi)=\mn(L,\chi)\cap \mm_{k,i}^a(d,\chi).$ 
\end{defn}
Since $L$ is $K_X$-negative, it is easy to see the following proposition.
\begin{prop}\label{bound}For fixed $(\chi,a)$, $\mm_{k,i}^a(L,\chi)$ is empty except for finitely many pairs $(k,i)$.
\end{prop}

\begin{defn}For two integers $l>0$ and $j$, we define $\mw_{l,j}^a(L,\chi)$ to be the (locally closed) substack of $\mm_{\bullet}^a(L,\chi)$ parametrizing sheaves $F\in\mm_{\bullet}^a(L,\chi)$ with  $h^0(F(-jK_X))=l$ and $h^0(F(-nK_X))=0,\forall n<j$. 

Let $\mv_{l,j}(L,\chi)=\mn(L,\chi)\cap \mw_{l,j}^a(L,\chi).$
\end{defn} 
\begin{rem}\label{duke}By sending each sheaf $F$ to its dual $\me xt^1(F,K_X)$, we get an isomorphism $\mm_{k,i}^a(L,\chi)\xrightarrow{\cong}\mw^{-\chi+a}_{k,-i}(L,-\chi)$, which identifies $\mn_{k,i}(L,\chi)$ with $\mv_{k,-i}(L,-\chi)$.
\end{rem}

\begin{prop}\label{dnki}For $\chi-iK_X.L\geq0$, $dim~\mn_{k,i}(L,\chi)\leq L^2-(\chi-iK_X.L)-k$.
\end{prop}
\begin{proof}Let $F\in\mn_{k,i}(L,\chi)$, then $H^1(F(-iK_X))\neq 0$ and hence we have a non split exact sequence
\begin{equation}\label{hib}0\ra K_X\ra I_F(L+K_X)\ra F(-iK_X)\ra 0.
\end{equation} 
Since $Supp(F)$ is integral and (\ref{hib}) does not split, $I_F\in Hilb^{[\tilde{d_i}]}(X)$ with $\tilde{d_i}:=\frac{L.(L+K_X)}2-(\chi-iK_X.L)$.  

On the other hand, let $I_{\tilde{d_i}}$ be an ideal sheaf of colength $\tilde{d_i}$, let $h\in\text{Hom}(K_X,I_{\tilde{d_i}}(L+K_X))$ with $h\neq0$, then $h$ has to be injective.  Let $F_h$ be the cokernel. 
\begin{equation}\label{bih}0\ra K_X\xrightarrow{h} I_{\tilde{d_i}}(L+K_X)\ra F_h\ra0.
\end{equation}

Denote by $\mh^{\tilde{d_i}}_{\chi-(1+i)K_X.L+1}$ the (locally closed) substack of $\mh^{\tilde{d_i}}$ parametrizing ideal sheaves $I_{\tilde{d_i}}$ such that $dim~H^0(I_{\tilde{d_i}}(L))=\chi-(1+i)K_X.L+1$.  By (\ref{hib}), $I_F\in\mh^{\tilde{d_i}}_{\chi-(1+i)K_X.L+1}$ if $F\in\mn_{k,i}(L,\chi)$.

Let $\mext^1(\mn_{k,i},K_X)^{*}$ be the stack over $\mn_{k,i}(L,\chi)$ parametrizing non-spliting extensions in $\text{Ext}^1(F(-iK_X),K_X)$ with $F\in\mn_{k,i}(L,\chi)$.  Then 
\[dim~\mext^1(\mn_{k,i},K_X)^{*}=k+dim~\mn_{k,i}(L,\chi)\]

Let $\mhom(K_X,\mh^{\tilde{d_i}}_{\chi-(i+1)K_X.L+1})^{*}$ be the stack over $\mh^{\tilde{d_i}}_{\chi-(i+1)K_X.L+1}$ parametrizing non-zero maps in $\text{Hom}(K_X,I_{\tilde{d_i}}(L+K_X))$ with $I_{\tilde{d_i}}\in \mh^{\tilde{d_i}}_{\chi-(i+1)K_X+1}$. Then
\[\begin{array}{r}dim~\mhom(K_X,\mh^{\tilde{d_i}}_{\chi-(i+1)K_X.L+1})^{*}=\chi-(i+1)K_X.L+1+dim~\mh^{\tilde{d_i}}_{\chi-(i+1)K_X.L+1}\\
\leq 2\tilde{d_i}+\chi-(i+1)K_X.L=L^2-(\chi-iK_X.L).\end{array}\]

We then have an injection by (\ref{hib})
\[\mext^1(\mn_{k,i},K_X)^{*}\hookrightarrow \mhom(K_X,\mh^{\tilde{d_i}}_{\chi-(i+1)K_X.L+1})^{*}.\]
Hence 
\[dim~\mext^1(\mn_{k,i},K_X)^{*}\leq dim~\mhom(K_X,\mh^{\tilde{d_i}}_{\chi-(i+1)K_X.L+1})^{*},\]
which implies
\[dim~\mn_{k,i}(L,\chi)\leq L^2-(\chi-iK_X.L)-k.\]
The proposition is proved.
\end{proof}
\begin{rem}\label{duck}By Proposition \ref{dnki} and Remark \ref{duke}, we know that \[dim~\mv_{l,j}(L,\chi)\leq L^2+(\chi-jK_X.L)-l,~for~\chi-jK_X.L<0.\]
\end{rem}

Now we start with an ideal sheaf $I_{\tilde{d}}$ and a nonzero element $h$ in $\text{Hom}(K_X,I_{\tilde{d}}(L+K_X))$, then by (\ref{bih}) this will give us a 1-dimensional sheaf $F_{h}$.  The following lemma states that $F_h$ is always pure. 

\begin{lemma}\label{quit}Let $J$ be any torsion free rank 1 sheaf on $X$ such that $H^0(J)\neq 0$.  Then any nonzero element  $h_J\in H^0(J)$ gives a sequence
\[0\ra\mo_{X}\xrightarrow{h_J} J\ra F_{h_J}\ra 0,\]
with $F_{h_J}$ pure of dimension one.
\end{lemma}
\begin{proof}The injectivity of $h_J$ is obvious.  Let $T\subset F_{h_J}$ be 0-dimensional.  Since $\text{Ext}^1(T,\mo_{X})^{\vee}\cong \text{Ext}^1(\mo_{X},T)=0$, $T$ must also be contained in $J$.  Then $T=0$ by the torsion freeness of $J$.  Hence the lemma.
\end{proof}

By (\ref{bih}), $h^0(I_{\tilde{d_i}}(L+K_X))=h^0(F_h)$.  We stratify $\mh^n$ via $h^0(I_n(L+K_X))$.
\begin{defn}Let $n=\frac{L.(L+K_X)}2+\Delta$ for some $\Delta>0$ such that $n>0$.  Let $\mh^{n,l}_L~(0\leq l\leq h^0(L+K_X))$ be the substack of $\mh^n$ parametrizing ideal sheaves $I_{n}$ of colength $n$ satisfying that $h^0(I_n(L+K_X))=l$.  
\end{defn}

We have the dimension estimate for $\mh^{n,l}_L$ as follows. 
\begin{lemma}\label{inch}If $h^0(L+K_X)\leq0$, then $\mh^n=\mh^{n,0}_L$.  If $L+K_X>0$, assume moreover $s_{L+K_X}\geq0$, then for $l>0$, $dim~\mh^{n,l}_L\leq 2n-1-\Delta$.  
\end{lemma}
\begin{proof}Obviously if $L+K_X\leq0$, then $h^0(I_n(L+K_X))=0$ for all $I_n$ with $n>0$.  Assume that $L+K_X>0$.  For an ideal sheaf $I_n\in\mh_L^{n,l}$ with $l>0$, we can fit it into the following sequence.
\[0\ra\mo_{X}\ra I_n(L+K_X)\ra F\ra 0.\]
By Lemma \ref{quit}, $F\in\mm_{\bullet}^a(L+K_X,-\Delta)$ (with $a=l$ for instance).  Moreover $h^0(F(K_X))\leq h^0(F)=l-1.$  Hence $dim~H^1(F(K_X))\leq l-1+\Delta-K_X.(L+K_X).$  Then by analogous argument to the proof of Proposition \ref{dnki}, we have
\[dim~\mh_L^{n,l}+l\leq dim~\mm_{\bullet}^a(L+K_X,\Delta)+l-1-K_X.(L+K_X)+\Delta=2n-1-\Delta+l,\]
where $dim~\mm^a_{\bullet}(L+K_X,\Delta)=(L+K_X)^2$ because $s_{L+K_X}\geq0$.  Hence the lemma.
\end{proof}

Let $\mu_A(-)$ be the $A$-valued motivic measure (see e.g. Section 1 in \cite{kap}) with $A$ a commutative ring or a field if needed.  Denote by $A_n$ the subgroup (not a subring) generated by the image of $\mu_A(\ms)$ with $dim~\ms\leq n$.  

By Proposition \ref{dlt}, we know that 
\[\mu_A(\mm^a_{\bullet}(L,\chi))~\equiv~\mu_A(\mm(L,\chi))~~~mod~(A_{L^2-s_{L}}).\]

If Theorem \ref{entire} applies to $L$, then we have
\[\mu_A(\mm^a_{\bullet}(L,\chi))~\equiv~\mu_A(\mn(L,\chi))~~~mod~(A_{L^2-\rho_L}).\]

For two numbers $\chi$ and $\chi'$, we say that $\chi\sim\chi'$ if $\exists~ \widehat{L}\in Pic(X)$ such that $\pm\chi\equiv\chi'$ $mod~(\widehat{L}.L)$.  It is easy to see that $\mn(L,\chi)\cong\mn(L,\chi')$ if $\chi\sim \chi'$.  Hence we may take 
$K_X.L\leq \chi<0$.

Since $K_X.L\leq\chi<0$, by Proposition \ref{dnki} and Remark \ref{duck} we have for a generic $[F]\in \mn(L,\chi)$, $h^0(F)=0$ and $h^1(F(-K_X))=0$.  Let $\tilde{d}=\tilde{d_0}=\frac{L.(L+K_X)}2-\chi$.  Then $\forall I_{\tilde{d}}\in\mh_L^{\tilde{d},0}$, $H^0(I_{\tilde{d}}(L))\neq0$ since $\chi(I_{\tilde{d}}(L))=-K_X.L+1+\chi>0.$  Define $\mh_L^{\tilde{d},0,0}$ to be the open substack of $\mh_L^{\tilde{d},0}$ parametrizing ideal sheaves $I_{\tilde{d}}\in\mh_L^{\tilde{d},0}$ such that $H^1(I_{\tilde{d}}(L))=0$.
\begin{lemma}\label{ddv}If Theorem \ref{entire} applies to $L$, then $\mh_L^{\tilde{d},0}-\mh_L^{\tilde{d},0,0}$ is of dimension $\leq 2\tilde{d}-1-\min\{\chi-K_X.L,\rho_L\}$.
\end{lemma}
\begin{proof}$\forall I_{\tilde{d}}\in\mh_L^{\tilde{d},0}-\mh_L^{\tilde{d},0,0}$, $H^0(I_{\tilde{d}}(L))\neq0$.  Hence by Lemma \ref{quit} we have the following exact sequence
\[0\ra K_X\ra I_{\tilde{d}}(L+K_X)\ra F\ra 0,\]
with $F\in\mm_{\bullet}^{a}(L,\chi)$.  

Since $H^0(F)\cong H^0(I_{\tilde{d}}(L+K_X))=0$, $h^1(F)=-\chi$.  Moreover since $I_{\tilde{d}}\in \mh^{\tilde{d},0}-\mh^{\tilde{d},0,0}$, $H^1(F(-K_X))\cong H^1(I_{\tilde{d}}(L))\neq 0$, hence $F\in\coprod_{i\geq1}\mm^a_{k,i}(L,\chi)$.  By Proposition \ref{dnki} and Theorem \ref{entire}, $dim~\coprod_{i\geq1}\mm^a_{k,i}(L,\chi)\leq L^2-\min\{\chi-K_X.L,\rho_L\}$.  By the analogous argument to the proof of Proposition \ref{dnki} we have
\[dim~(\mh_L^{\tilde{d},0}-\mh_L^{\tilde{d},0,0})-K_X.L+1+\chi\leq L^2-\min\{\chi-K_X.L,\rho_L\}-\chi,\]
since $2\tilde{d}=L(L+K_X)-2\chi$.  Hence the lemma. 
\end{proof}

For every sheaf $F\in\mm_{\bullet}^a(L,\chi)$, there is a non split sequence
\begin{equation}\label{nsp}0\ra K_X\ra \widetilde{I}\ra F\ra0.\end{equation}
$\widetilde{I}$ can have torsion if $F\not\in \mn(L,\chi)$.  If $\widetilde{I}$ is torsion free, then $\widetilde{I}\cong I_{\tilde{d}}(L+K_X)$ for some ideal sheaf $I_{\tilde{d}}$ with colength $\tilde{d}=\frac{L.(L+K_X)}2-\chi$.  Let $\um^a(L,\chi)$ be the open substack of $\mm^a_{\bullet}(L,\chi)$ parametrizing sheaves $F$ such that $H^0(F)=0$ and $H^1(F(-K_X))=0$.  Then we have
\[\mm^a_{\bullet}(L,\chi)=\um^a(L,\chi)\cup(\coprod_{j\leq 0} \mw^a_{l,j}(L,\chi)\cup\coprod_{i\geq 1}\mm^a_{k,i}(L,\chi)).\]
By Proposition \ref{dnki} and Remark \ref{duck} we have
\[dim~(\coprod_{j\leq 0} \mw^a_{l,j}(L,\chi)\cup\coprod_{i\geq 1}\mm^a_{k,i}(L,\chi))\leq L^2-\min\{\rho_L,-\chi,\chi-K_X.L\}.\]
Hence
\[\mu_A(\mm^a_{\bullet}(L,\chi))~\equiv~\mu_A(\um^a(L,\chi))~~~mod~(A_{L^2-\min\{\rho_L,-\chi,\chi-K_X.L\}}).\]

Define $\mn_0(L,\chi):=\mn(L,\chi)\cap\um^a(L,\chi)$.  Then 
\begin{equation}\label{mmm}\begin{array}{l}\mu_A(\mm_{\bullet}^a(L,\chi))\equiv\mu_A(\mm(L,\chi))\equiv\mu_A(\um^a(L,\chi))\\ \qquad\qquad\qquad\equiv\mu_A(\mn_0(L,\chi))~~~mod~(A_{L^2-\min\{\rho_L,-\chi,\chi-K_X.L\}}).\end{array}\end{equation}

Lemma \ref{inch} and Lemma \ref{ddv} together imply that
\begin{equation}\label{mmh}\mu_A(\mh^{\tilde{d}})~\equiv~\mu_A(\mh_L^{\tilde{d},0})~\equiv~\mu_A(\mh_L^{\tilde{d},0,0})~~~mod~(A_{2\tilde{d}-1-\min\{\rho_L,-\chi,\chi-K_X.L\}}).
\end{equation}

Let $\mext^1(-,K_X)^{*}$ and $\mhom(K_X,-)^{*}$ be as defined in the proof of Proposition \ref{dnki}.  The sequence (\ref{nsp}) induces a birational map
\[\theta:\mext^1(\mm^a_{\bullet}(L,\chi),K_X)^{*}\dashrightarrow \mhom(K_X,\mh^{\tilde{d}})^{*}.\]
$\theta$ is surjective for $a$ big enough.  

Denote by $\mathbb{U}^a(L,\chi)$ the preimage of $\mhom(K_X,\mh^{\tilde{d},0,0})^{*}$ via $\theta$.  Then we have
\begin{equation}\label{bird}\mu_A(\mathbb{U}^a(L,\chi))=(\bl^{-K_X.L+1+\chi}-1)\cdot\mu_A(\mh_L^{\tilde{d},0,0}),\end{equation}
where $\bl:=\mu_A(\mathbb{A})$ with $\mathbb{A}$ the affine line.  Then by (\ref{mmh}) we have
\begin{equation}\label{due}\begin{array}{l}\mu_A(\mathbb{U}^a(L,\chi))~\equiv~(\bl^{-K_X.L+1+\chi}-1)\cdot\mh^{\tilde{d}}~\\ \qquad\qquad\qquad\equiv~\bl^{-K_X.L+1+\chi}\cdot\mh^{\tilde{d}}~~~~mod~(A_{L^2-\chi-\min\{\rho_L,-\chi,\chi-K_X.L\}})\end{array}
\end{equation}

On the other hand, we have \[\mext^1(\mn_0(L,\chi),K_X)^{*}\subset \mathbb{U}^a(L,\chi)\subset \mext^1(\um^a(L,\chi),K_X)^{*}.\]
Hence by (\ref{mmm}),
\begin{equation}\label{une}\begin{array}{l}\mu_A(\mathbb{U}^a(L,\chi))~\equiv~(\bl^{-\chi}-1)\cdot\mu_A(\mn_0(L,\chi))~\\ \qquad\qquad\qquad\equiv~(\bl^{-\chi}-1)\cdot\mu_A(\mn(L,\chi))\\
\qquad\qquad\qquad\equiv~\bl^{-\chi}\cdot\mu_{A}(\mn(L,\chi))\\
\qquad\qquad\qquad\equiv~\bl^{-\chi}\cdot\mu_{A}(\mm(L,\chi))~~~mod~(A_{L^2-\chi-\min\{\rho_L,-\chi,\chi-K_X.L\}}).\end{array}
\end{equation}
Combine (\ref{due}) and (\ref{une}), we have our main theorem as follows.
\begin{thm}\label{main}Assume Theorem \ref{entire} applies to $L$ and moreover either $L+K_X\leq0$ or $s_{L+K_X}\geq0$.  For any $\chi$, let $\chi_0\sim\chi$ and $K_X.L\leq\chi_0<0$.  Then we have
\[\mu_A(\mm(L,\chi))~\equiv~\bl^{-K_X.L+1+2\chi_0}\cdot \mu_A(\mh^{\tilde{d}}),~~~mod~(A_{L^2-\min\{\rho_L,-\chi_0,\chi_0-K_X.L\}}),\]
with $\tilde{d}=\frac{L.(L+K_X)}2-\chi_0$ and $\rho_L$ defined in Theorem \ref{entire}.  

On the scheme level we have
\[\mu_A(M(L,\chi))\equiv\bl^{-K_X.L+1+2\chi_0}\cdot \mu_A(Hilb^{[\tilde{d}]}(X))~~~mod~(A_{L^2+1-\min\{\rho_L,-\chi_0,\chi_0-K_X.L\}}).\]
\end{thm}


The Betti numbers and Hodge numbers of $Hilb^{[n]}(X)$ are well known (e.g. see \cite{gs}).  Theorem \ref{main} implies that we can get some virtual Hodge numbers and virtual Betti numbers of $M(L,\chi)$.
\begin{coro}\label{bet}Let $b^{(v)}_i(-)$ and $h_{(v)}^{p,q}(-)$ be the $i$-th (virtual) Betti number and (virtual) Hodge number with index $(p,q)$ respectively.  Assume Theorem \ref{main} applies to $(X,L)$.  Let $\tilde{d}$ be the same as in Theorem \ref{main}.  Then for $i$ and $p+q$ no less than $1+2(L^2+1-\min\{\rho_L,-\chi_0,\chi_0-K_X.L\})$, we have

(1) $b^v_{i}(M(L,\chi))=0$ for $i$ odd.

(2) $b^v_{2p}(M(L,\chi))=b_{2p-2(1+2\chi_0-K_X.L)}(Hilb^{[\tilde{d}]}(X))$. 

(3) $h_v^{p,q}(M(L,\chi))=h^{p-(1+2\chi_0-K_X.L),q-(1+2\chi_0-K_X.L)}((Hilb^{[\tilde{d}]}(X)))$.

If moreover $M^{ss}(L,\chi)=M(L,\chi)$, then $b^v_i(M(L,\chi))=b_i(M(L,\chi))$ and $h^{p,q}_v(M(L,\chi))=h^{p,q}(M(L,\chi))$.
\end{coro}
\begin{coro}\label{star}If there is a universal sheaf over $M(L,\chi)$,  then $M(L,\chi)$ is stably rational, i.e. $\exists~ S$ a rational scheme, such that $M(L,\chi)\times S$ is rational.
\end{coro}
\begin{proof}Let $N_0(L,\chi)$ and $Hilb^{[\tilde{d}],0,0}(X)$ be the scheme associated to $\mn_0(L,\chi)$ and $\mh^{\tilde{d},0,0}$ respectively.  Let $\mf$ be a universal sheaf over $N_0(L,\chi)$.  Let $\mathcal{I}_{\tilde{d}}$ be the universal ideal sheaf over $Hilb^{[\tilde{d}],0,0}(X)$.  $p$ ($q$, resp.) is the projection from $X\times M$ to $M$ ($X$, resp.) for $M=N_0(L,\chi)$ or $Hilb^{\tilde{d},0,0}(X)$.  We can see that the projective bundle $\mathbb{P}(\mathcal{E}xt_p^1(\mf,K_X))$ over $N_0(L,\chi)$ is birational to the projective bundle $\mathbb{P}(\mathcal{H}om_p(K_X,\mathcal{I}_{\tilde{d}}(q^{*}(L\otimes K_X))))$ over $Hilb^{[\tilde{d}],0,0}(X)$ which is rational for $X$ rational.  Hence the corollary.
\end{proof}
\begin{rem}By Theorem 3.19 in \cite{lee}, Corollary \ref{star} applies to $M(dH,\chi)$ over $X=\p^2$ such that $d$ and $\chi$ are coprime.  By Proposition 4.5 in \cite{yth}, $M(dH,\chi)$ is rational for $\chi\equiv \pm1~mod~(d)$.
\end{rem}

\section{The case $X=\p^2$.}
Theorem \ref{main} applies to many examples on $\p^2$ or Hirzebruch surfaces.  In this section we let $X=\p^2$ and $L=dH$, and we then obtain some explicit results.  For any $\chi$, $\chi_0$ in Theorem \ref{main} can be chosen to satisfy $-2d-1\leq\chi_0\leq-d+1$.
Recall that in this case $\rho_d=\rho_{dH}$ can be chosen as follows.
\[\rho_d:=\left\{\begin{array}{l}d-1,~for~d=p~or~2p~with~p~prime.\\ \\ 7,~otherwise.\end{array}\right.\]
Hence $\min\{\rho_d,-\chi_0,\chi_0+3d\}=\rho_d$.
\begin{coro}\label{disk}For any $d>0$ and $\chi_1,\chi_2$, we have
\[\mu_A(\mm(dH,\chi_1))~\equiv~ \mu_A(\mm(dH,\chi_2)),~~~mod~(A_{d^2-\rho_d}).\]
On the scheme level we have
 \[\mu_A(M(dH,\chi_1))~\equiv~ \mu_A(M(dH,\chi_2)),~~~mod~(A_{d^2+1-\rho_d}).\]
\end{coro}
\begin{proof}By Theorem \ref{main} the corollary is equivalent to say that for any $-2d-1\leq \chi_1,\chi_2\leq -d+1$,
\begin{equation}\label{hi}\bl^{3d+1+2\chi_1}\cdot \mu_A(\mh^{n_1})~\equiv~\bl^{3d+1+2\chi_2}\cdot \mu_A(\mh^{n_2}),~~~mod~(A_{d^2-\rho_d}),\end{equation}
where $n_i=\frac{d(d-3)}2-\chi_i$.

It is enough to show (\ref{hi}) for $\chi_1=-2d-1$ and $\chi_2=-d+1$ which follows from $M(d,-2d-1)\cong M(d,-d+1)$.  Hence the corollary.
\end{proof}
\begin{rem}If $d=p$ or $2p$ with $p$ prime, then the codimension $d-1$ can not be sharpened, i.e. in general
\[\mu_A(\mm(dH,\chi))~\not\equiv~\bl^{3d+1+2\chi_0}\cdot \mu_A(\mh^{\frac{d(d-3)}2-\chi_0}),~~~mod~(A_{d^2-d}).\]
We can see this from the examples $d=4$ and $d=5$ computed in \cite{yth}. 
\end{rem}

\begin{rem}\label{pass}For $d$ and $\chi$ not coprime, $M^{ss}(dH,\chi)-M(dH,\chi)$ is not empty.  But the S-equivalence classes of strictly semistable sheaves form a closed subset of codimension $\geq d-1$ in $M^{ss}(dH,\chi)$.  Hence we still have
\[\mu_A(M^{ss}(dH,\chi))~\equiv~\bl^{3d+1+2\chi_0}\cdot \mu_A(Hilb^{[\frac{d(d-3)}2-\chi_0]}(\p^2)),~~~mod~(A_{d^2-\rho_d+1}).\]
However, since $M^{ss}(d,\chi)$ might not be smooth, we only have similar conclusion to Corollary \ref{bet} on its virtual Betti numbers. 
\end{rem}
At the end we write down the following theorem as an easy corollary to Corollary \ref{bet}, Corollary \ref{disk}, Remark \ref{pass} and the well-known fact on the Betti numbers of $Hilb^{[n]}(\p^2)$.
\begin{thm}\label{proj}
Let $X=\p^2$ with $H$ the hyperplane class.  Let $b_i$ be the $i$-th Betti number of $M^{ss}(dH,\chi)$.  If $d\geq 8$ and $M^{ss}(dH,\chi)$ is smooth, then we have

(1) $b_0=1,~b_2=2,~b_4=6,~b_6=13,~b_8=29,~b_{10}=57,~b_{12}=113;$

(2) $b_{2i-1}=0$ for $i\leq7$.  

(3) For $p+q\leq13$, $h^{p,q}=b_{p+q}\cdot\delta_{p,q}$, where $\delta_{p,q}=\left\{\begin{array}{l}1,~for~p=q.\\ \\ 0,~otherwise.\end{array}\right..$
\end{thm}

\section*{\huge{Appendix.}}
\appendix
\section{The proof of Theorem \ref{tt}.}  
We give a whole proof of Theorem \ref{tt} in this section.  We state the theorem again here.
\begin{thm}[Theorem \ref{tt}]Let $L.K_X<0$ and $g_{\frac Lk}>0$.  Then $dim~\mt_{2}^o(L,\chi)\cap \mc_{\frac Lk}\leq L^2+(K_X.L+1)$.
\end{thm}
\begin{proof}Recall that we have defined $\mc_{\frac Lk,a}$ in $\mm^a_{\bullet}(L,\chi)$.  Let $\mt_{n,a}^o(L,\chi)$ be the analog of $\mt_n^o(L,\chi)$ in $\mm_{\bullet}^a(L,\chi)$.  Let $F\in\mt_{2,a}^o(L,\chi)\cap\mc_{\frac Lk,a}$ with lower and upper filtrations $\{F_i\}$ and $\{F^i\}$ (see Definition \ref{uplow}) with factors $\{Q_i\}$ and $\{R_i\}$ respectively.  Let $m$ be the length of the two filtrations.  Then $t_m=r_1\leq 2$ by Remark \ref{fib}.  If $r_1=2$, then $R_m\cong F\otimes\mo_C$ has to be locally free of rank 2.  Since $g^m_F:R_m\twoheadrightarrow R_{m-1}$ is surjective, $R_{m-1}$ is either of rank 1 or locally free of rank 2 and if $R_{m-1}$ is locally free of rank 2, $g^m_F$ is an ismorphism.   

On the other hand, $g_{\frac Lk}>0$ implies that $K_X.L\geq -\frac{L^2}k$.  By the similar argument to Proposition \ref{dlt}, we can get
\begin{eqnarray}\label{useful}&&dim~(\mt_{2,a}^o(L,\chi)\cap \mc_{\frac Lk,a}-\mt_{2}^o(L,\chi)\cap \mc_{\frac Lk})\qquad\qquad\nonumber\\ &\leq& L^2-\min_{\footnotesize {\begin{array}{c}\sum_i l_i=k;\\\forall i,0<l_i<k\end{array}}}(\sum_{i<j}l_il_j)\cdot \frac{L^2}{k^2}\nonumber\\&=&L^2-\frac{k-1}{k^2}L^2\leq L^2+\frac{k-1}kK_X.L.\end{eqnarray}

We prove the theorem case by case.  

\emph{Case 1.} $r_1=1$.  Then by Lemma \ref{rko} we are done. 

\emph{Case 2.} $r_i=2$ for all $1\leq i\leq m$.  

If $F\in \mc_{\frac Lk}^{2,\cdots,2}\cap\mt_{2}^o(L,\chi)$, then $R_i\cong Q_i\cong R_m(-(m-i)C)$ and the two filtrations coincide with all  factors locally free of rank 2.  In this case $k=2m$.  Let $R:=R_m$.  Then $c_1(R)=\frac Lm$ and we have 
\begin{equation}\label{chir}\sum_{i=0}^{m-1}\chi(R_m(-iC))=m\cdot\chi(R_m)-\frac{(m-1)m}4\cdot(\frac{L}{m})^2=\chi.\end{equation}
Hence $\chi(R)$ is fixed by $(L,\chi,k)$.  By the stability of $F$, we have
\begin{equation}\label{sir}\forall ~I\subset R ~of~rank~1,~\chi(I)< \frac{\chi}{2m}+\frac{m-1}{4m^2}L^2=\frac{\chi(R_m)}2+\frac{m-1}{8m^2}L^2.\end{equation}  

Let $\mr$ be the parametrizing stack of such $R$.  We want to show that 
\begin{equation}\label{dpr}dim~\mr\leq\frac{L^2}{m^2}+(\frac {K_X.L}m+1).\end{equation}

With no loss of generality, we assume $0<\chi(R)\leq -\frac{K_X.L}m$, then we have the following exact sequence.
\begin{equation}\label{two}0\ra\mo_C\ra R\ra I_2\ra0.
\end{equation}
By the same argument as in Proposition \ref{coha}, we get the equation in (\ref{dpr}). 

Now we do the induction. Let $\mathcal{P}_{F/F_1}$  be the parametrizing stack of $F/F_1=F/R(-(m-1)C)$.  Then by (\ref{useful}) and the induction assumption we have
\begin{equation}\label{pff} dim~\mathcal{P}_{F/F_1}\leq
\frac{L^2(m-1)^2}{m^2}+\frac{m-\frac32}mK_X.L\end{equation}

$dim~\text{Ext}^2(F/F_1,F_1)=dim~\text{Hom}(F_1,F_2/F_1(K_X))=dim~Hom(R,R(K_X+C))$.   We want to find a upper bound $N$ of $dim~\text{Hom}(R, R(K_X+C)).$  Notice that $dim~\text{Hom}(R,R(K_X+C))\leq dim~\text{Hom}(R,R)+4(2g_C-2)$  If $R$ is stable, then $\text{Hom}(R,R)\cong\mathbb{C}$.  If $R$ is not stable, then by (\ref{sir}) we have $dim~\text{Hom}(R,R)\leq 3+\frac {m-1}{4m^2}L^2.$  Therefore,
\begin{eqnarray}dim~\text{Hom}(R, R(K_X+C))\leq 3+\frac{m-1}{4m^2}L^2+4(2g_C-2)\qquad\qquad&&\nonumber\\
= 3+\frac{m-1}{4m^2}L^2+4(\frac{K_X.L}{2m}+\frac{L^2}{4m^2}) =:N&&
\end{eqnarray}

Then we have 
\begin{eqnarray}\label{fcot}dim~\mc^{2,\cdots,2}_{\frac Lk}\cap\mt_2^o(L,\chi)&\leq& dim~\mathcal{P}_{F/F_1}+N-\chi(F/F_1,F_1)\nonumber\\
&\leq& \frac{L^2(m-1)^2}{m^2}+\frac{(m-\frac32)}mK_X.L+N+\frac{L^2(m-1)}{m^2}\nonumber\\
&=&  L^2+(K_X.L+1)+\frac{K_X.L}{2m}+(2-\frac{3(m-1)}{4m^2}L^2)\nonumber\\
&\leq&  L^2+(K_X.L+1).
\end{eqnarray}
The last equation is because $K_X.\frac L{2m}\in\mathbb{Z}_{<0}$ and $(\frac{L}{2m})^2\in\mathbb{Z}_{>0}$.  

Now we compute the dimension of  $\mc^{2,\cdots,2,1,\cdots,1}_{\frac Lk}\cap\mt_2^o(L,\chi)$.  We do the induction on the number $\ell(1)$ of $1$ in the superscript of $\mc^{2,\cdots,2,1,\cdots,1}_{\frac Lk}$.   

\emph{Case 3.} $\ell(1)=1$.  

Let $F\in \mc^{2,\cdots,2,1}_{\frac Lk}\cap\mt_2^o(L,\chi)$.  Let $k=2m-1$ with $m\geq 2$.  Let $C$ be its reduced support and hence $C\in |\frac{L}{2m-1}|^{int}$.  We take the lower and upper filtrations $\{F_i\}$ and $\{F^i\}$ of $F$ with factors $\{Q_i\}$ and $\{R_i\}$ for $1\leq i\leq m$.  Then $R_m$ is a rank 2 bundle on $C$, $R_{i}\cong R_m((-m+i)C)$ for $2\leq i\leq m$ and $R_1$ is a rank 1 torsion free sheaf on $C$ with surjection $g^2_F:R_2(-C)\twoheadrightarrow R_1$.  Let $K$ be the kernel of $g^2_F$, then $K$ is torsion free of rank 1.  We have an exact sequence 
\begin{equation}\label{suf}0\ra K\ra R_2(-C)\ra R_1\ra 0.\end{equation} 

$K((m-1)C)$ is a subsheaf of $R_m$.  By the stability of $F$, we know that
 \begin{eqnarray}\label{kk}&&\chi(F^{m-1})+\chi(K(m-1)C)=\sum_{i=1}^{m-1}\chi(R_i)+\chi(K((m-1)C))\nonumber\\
 &=&(m-1)(\chi(R_1)+\chi(K))+\sum_{i=1}^{m-2}\frac {2iL^2}{(2m-1)^2}+\frac{(m-1)L^2}{(2m-1)^2}\nonumber\\ 
 &<&\frac{(2m-2)\chi}{2m-1}.\end{eqnarray}
(\ref{kk}) 
implies that 
\begin{equation}\label{kb}\chi(R_2)-\frac{2L^2}{(2m-1)^2}=\chi(R_1)+\chi(K)\leq \frac{2\chi}{2m-1}-\frac{(m-1)L^2}{(2m-1)^2}.
\end{equation}
Since $R_1$ is a quotient of $R_2(-C)$, $R_1((m-1)C)$ is a quotient of $R_m$ hence a quotient of $F$.  So 
\begin{equation}\label{cro}\chi(R_1)+\frac{(m-1)L^2}{(2m-1)^2}>\frac {\chi}{2m-1}\Leftrightarrow \chi(R_1)> \frac{\chi}{2m-1}-\frac{(m-1)L^2}{(2m-1)^2}.\end{equation}
Combine (\ref{kb}) 
and (\ref{cro}), then we get 
\begin{equation}\label{drb}\chi(K)-\chi(R_1)\leq \frac{(m-1)L^2}{(2m-1)^2},
\end{equation}
We need a upper bound for $dim~\text{Ext}^2(F/R_1,R_1)=dim~\text{Hom}(R_1,F/R_1(K_X))$.  The upper and lower filtrations of $F/R_1$ coincide.  Hence $\text{Hom}(R_1,F/R_1(K_X))=\text{Hom}(R_1,R_2(K_X))$.  Then we have 
\begin{eqnarray}\label{rrg}&&dim~\text{Ext}^2(F/R_1,R_1)=dim~\text{Hom}(R_1,R_2(K_X))\nonumber\\
&\leq& dim~\text{Hom}(R_1,R_1(K_X+C))+dim~\text{Hom}(R_1,K(K_X+C))\nonumber\\
&\leq& 4g_C-2+\chi(K)-\chi(R_1).
\end{eqnarray} 
By (\ref{drb}) 
we have
\begin{equation}\label{lei}dim~\text{Ext}^2(F/R_1,R_1)\leq N:=\frac{(m-1)L^2}{(2m-1)^2}+4g_C-2.
\end{equation}

Let $\mathcal{P}_{F/R_1}$ be the parametrizing stack of $F/R_1$.  Then $F/R_1\in\mc_{\frac{L}{2m-2},a}^{2,\cdots,2}\cap\mt^{o}_{2,a}(\frac{(2m-2)L}{2m-1},\widetilde{\chi})$.  Assume first $m\geq3$, then by Case 2 and (\ref{useful}), we have


\begin{equation}\label{dfro}dim~\mathcal{P}_{F/R_1}\leq (\frac {(2m-2)L}{2m-1})^2+\frac {(2m-3)L.K_X}{2m-1}.\end{equation}  

Hence by standard argument we have 
\begin{eqnarray}\label{exo}&dim&\mc^{2,\cdots,2,1}_{\frac Lk}\cap\mt_2^o(L,\chi)\leq dim~\mathcal{P}_{F/R_1}+g_C-1+N-\chi(F/R_1,R_1)\nonumber\\
&\leq&(\frac {(2m-2)L}{2m-1})^2+\frac {(2m-3)L.K_X}{2m-1}+g_C-1+N+\frac{(2m-2)L^2}{(2m-1)^2}\nonumber\\
&=& L^2+(K_X.L+1)+(1-g_C)+\frac{L.K_X}{2m-1}-\frac{(m-3)L^2}{(2m-1)^2}+1\nonumber\\&\leq& L^2+(K_X.L+1)~for~m\geq3.
\end{eqnarray}

Let $m=2$, then $F/R_1=R_2$ and for fixed $K$ and $R_1$, $R_2$ is given by (\ref{suf}).  Hence we have 
\begin{eqnarray}\label{ear}&&dim~\mc^{2,1}_{\frac L3}\cap\mt_2^o(L,\chi)\nonumber\\&=&dim~|\frac L3|+2(g_C-1)-\chi(R_1,K)-\chi(R_2,R_1)\nonumber\\&&+dim~\text{Hom}(K,R_1(K_X))+dim~\text{Hom}(R_1,R_2(K_X))\nonumber\\
&\leq&dim~|\frac L3|+2(g_C-1)-\chi(R_1,K)-\chi(R_2,R_1)+dim~\text{Hom}(K,R_1(K_X))\nonumber\\&&\small{+dim~\text{Hom}(R_1,K(K_X+C))+dim~\text{Hom}(R_1,R_1(K_X+C))}\nonumber\\
&\leq& \frac 12(\frac L3)^2-\frac{K_X.L}6+2(g_C-1)+\frac{L^2}9+\frac{2L^2}9+4g_C-2+\frac{K_X.L}3+1\nonumber\\
&=&L^2+K_X.L+1-\frac{5L^2}{18}+2+\frac{K_X.L}6\nonumber\\
&=&\small{L^2+K_X.L+1+1-g_C+2+\frac{K_X.L}3-\frac {2L^2}9\leq L^2+K_X.L+1.}
\end{eqnarray}
Hence we are done for $\ell(1)=1$.

\emph{Case 4: The last case.} $\ell(1)\geq 2$.   

Let $F\in\mc_{\frac Lk}^{2,\cdots,2,1,\cdots,1}\cap\mt_{2}^o(L,\chi) $ with $\ell(1)\geq 2$.   Let $m_i=\ell(i)$ for $i=1,2$.  Then $m_1\geq2$ and $k=m_1+2m_2\geq4$.  Let $C$ be the reduced support of $F$.  $g_C>0$.  By doing the upper filtration, we can write $F$ into the following sequence
\begin{equation}0\ra F'\ra F\ra F''\ra 0,
\end{equation}  
with $F'\in\mc_{\frac Lk,a'}^{1,\cdots,1}$ and $F''\in\mc_{\frac{L}{k},a''}^{2,\cdots,2}\cap \mt_{2,a''}^{o}(\frac{2m_2}{m_1+2m_2}L,\chi'').$   

Take the upper and lower filtrations of $F'$ with graded factors $\{R'_i\}$ and $\{Q'_i\}$.  Then both $R'_i$ and $Q'_i$ are of rank 1.  Denote by $R'^{tf}_i$ the quotient of $R'_i$ module its torsion.  Then $Q'_{m_1}=R'^{tf}_{m_1}$.  

We know that the upper and lower filtrations of $F''$ coincide.  Let $R''_i$ be the factors.  Then $\{R''_i,R'_i\}$ is the set of graded factors of the upper filtration for $F$ and hence we have a surjection $g_F^{m_1+1}:R''_1(-C)\twoheadrightarrow R'_{m_1}$.  Hence we have a surjection $p^1_{m_1+1}:R''_1(-C)\twoheadrightarrow Q'_{m_1}$ as $Q'_{m_1}$ is a quotient of $R'_{m_1}$.  Let 
$K_{m_1}$ be the kernel of 
$p^1_{m_1+1}$.
\begin{equation}\label{kqro}0\ra K_{m_1}\ra R''_1(-C)\ra Q'_{m_1}\ra0.
\end{equation}
Denote by 
$P_{m_1}$ the subsheaf of $F/F'_{m_1-1}$ given by the following extension.
\begin{equation}\label{pqro}0\ra Q'_{m_1}\ra P_{m_1}\ra K_{m_1}(C)\ra 0.
\end{equation}
Then $P_{m_1}$ is a $\mo_C$-module, i.e. it is a rank 2 torsion free sheaf on $C$.  This is because $p^1_{m_1+1}$ is defined by acting $\delta_C$ on $F/F'_{m_1-1}$ and $K_{m_1}$ is the kernel which implies $\delta_C\cdot P_{m_1}=0.$  Moreover,  $P_{m_1}$ is the maximal subsheaf of $F/F'_{m_1-1}$ annihilated by $\delta_C$, since $Q'_{m_1}$ is torsion free of rank 1.   

Again we have a map $p^1_{m_1}:P_{m_1}(-C)\ra Q'_{m_1-1}$ inducing the injection $f_{F'}^{m_1}:Q'_{m_1}(-C)\hookrightarrow Q'_{m_1-1}.$  The map $p^1_{m_1}$ might not be surjective and we denote by $S'_{m_1-1}$ its image in $Q'_{m_1-1}$.  We have $Q'_{m_1}(-C)\subset S'_{m_1-1}\subset Q'_{m_1-1}$.  Let $S'_{m}=Q'_m$.  

Let $K_{m_1-1}$ be the kernel of $p^1_{m_1}$, then 
\begin{equation}\label{kcc}\chi(K_{m_1})+\chi(Q'_{m_1}(-C))-\chi(Q'_{m_1-1})\leq\chi(K_{m_1-1})\leq \chi(K_{m_1}).
\end{equation}   

Again we have a subsheaf $P_{m_1-1}$ of $F/F'_{m_1-2}$ such that $P_{m_1-1}$ is a rank 2 torsion free sheaf on $C$ lying in the following exact sequence.
\begin{equation}\label{pqrt}0\ra Q'_{m_1-1}\ra P_{m_1-1}\ra K_{m_1-1}(C)\ra0.
\end{equation}

We repeat this procedure to define $K_i$, $P_i$ and $S'_i$ for $1\leq i\leq m_1$, and finally we get
\begin{equation}\label{pqrl}0\ra Q'_{1}\ra P_{1}\ra K_1(C)\ra0.
\end{equation}
Since $P_1$ is rank 2 and $F/P_1$ is torsion-free on $C$, $P_1=F_1$ with $\{F_i\}$ the lower filtration of $F$.  

By (\ref{kcc}), (\ref{kqro}), (\ref{pqro}) and the recursion on $i$, we have $\forall~1\leq i\leq m_1-1,$
\begin{equation}\label{alli}\chi(K_i)+\chi(Q'_i)\geq \chi(K_{i+1})+\chi(Q'_{i+1})-C.C\geq \chi(K_{m_1})+\chi(Q'_{m_1})-(m_1-i)C.C,\end{equation} 

On the other hand, by Statement (3) in Lemma \ref{cop}, we have 
\begin{eqnarray}\label{pot}&&\chi(P_1)=\chi(R''_{m_2})-(m_1+2m_2-2)C^2\nonumber\\&\Rightarrow& \chi(K_1)+\chi(Q'_1)= \chi(K_{m_1})+\chi(Q'_{m_1})-(m_1-1)C.C.\end{eqnarray}

Combine (\ref{alli}) and (\ref{pot}), then we have $\forall~1\leq i\leq m_1-1$,
\begin{equation}\label{ally}\chi(K_i)+\chi(Q'_i)= \chi(K_{i+1})+\chi(Q'_{i+1})-C.C= \chi(K_{m_1})+\chi(Q'_{m_1})-(m_1-i)C.C.\end{equation} 
Hence $S'_{i}=Q'_i$ for all $1\leq i\leq m_1$.

Let $G_{(1)}=F/P_1$, then $G_{(1)}\in\mc_{\frac Lk,b_1}^{2,\cdots,2,1,\cdots,1}\cap\mt_{2,b}^o(\frac {m_1+2m_2-2}{m_1+2m_2}L,\chi_{(1)}) $ with $\ell(1)=m_1,$ $\ell(2)=m_2-1$.  Also we can write $G_{(1)}$ into the following sequence
\begin{equation}0\ra G'_{(1)}\ra G_{(1)}\ra G''_{(1)}\ra 0,
\end{equation}  
with $G'_{(1)}\in\mc_{\frac Lk,b'_1}^{1,\cdots,1}$ and $G''_{(1)}\in\mc_{\frac{L}{k},b''_1}^{2,\cdots,2}\cap \mt_{2,b''_1}^{o}(\frac{2m_2-2}{m_1+2m_2}L,\chi''_{(1)}).$   We see that $\chi(G''_{(1)})=\sum_{i=2}^{m_2}\chi(R''_i)$ and $\chi(G'_{(1)})=\sum_{i=1}^{m_1}\chi(S'_i(C))= \chi(F')+\frac{m_1L^2}{(m_1+2m_2)^2}$.  

We do the same procedure to $G_{(1)}$ as we did to $F$ and we can get $G_{(2)}$ with $\ell(1)=m_1$ and $\ell(2)=m_2-2$.  After $m_2$ steps, we finally get $G_{(m_2)}\in\mc^{1,\cdots,1}_{\frac Lk,b_{m_2}}$.  $G_{(m_2)}$ is a quotient of $F$.  Moreover by the stability of $F$, we have $\frac{m_1\chi}{m_1+2m_2}<\chi(G_{(m_2)})\leq \chi(F')+\frac{m_1m_2L^2}{(m_1+2m_2)^2}$.  Therefore we have
\begin{eqnarray}\label{sat}m_1\chi(Q'_1)&+&\sum_{i=1}^{m_1-1}\frac{L^2}{(m_1+2m_2)^2}\geq \chi(F')>\frac{m_1\chi}{m_1+2m_2}-\frac{m_1m_2L^2}{(m_1+2m_2)^2}\nonumber\\
\Rightarrow \chi(Q'_1)&>&\frac{\chi}{m_1+2m_2}-\frac{(2m_2+m_1-1)L^2}{2(m_1+2m_2)^2}.\end{eqnarray}

We also have 
\begin{eqnarray}\label{satt}&&\chi-\chi(F')=\chi(F'')=(\chi(Q'_{m_1})+\chi(K_{m_1}))m_2+\sum_{i=1}^{m_2}\frac{2iL^2}{(m_1+2m_2)^2}\nonumber\\
&\Rightarrow& \chi(Q'_{m_1})+\chi(K_{m_1})<\frac{2\chi}{m_1+2m_2}+\frac{m_1-m_2-1}{(m_1+m_2)^2}L^2\nonumber\\
&\Rightarrow&  \chi(Q'_{1})+\chi(K_{1})<\frac{2\chi}{m_1+2m_2}-\frac{m_2}{(m_1+m_2)^2}L^2.\end{eqnarray}
Combine (\ref{sat}) and (\ref{satt}) we have
\begin{equation}\label{fast}\chi(K_1)-\chi(Q'_1)<\frac{(m_1+m_2-1)L^2}{(m_1+2m_2)^2}.\end{equation}

Let $\mathcal{P}_{F/Q'_1}$ be the parametrizing stack of $F/Q'_1$.  By (\ref{useful}) and the induction assumption on $\ell(1)$, we have 
\begin{equation}\label{time}dim~\mathcal{P}_{F/Q'_1}\leq \frac {(m_1+2m_2-1)^2L^2}{(m_1+2m_2)^2}-\frac{(m_2+2m_2-2)}{m_1+2m_2}K_X.L.\end{equation}
Let $\widetilde{F}_1$ be the maximal subsheaf of $F/Q'_1$ annihilated by $\delta_C$.  Then we have the following sequence.
\[0\ra K_1(C)\ra \widetilde{F}_1\ra S'_1(C)\ra 0.\]
Notice that $S'_1=Q'_1$.  Hence
\begin{eqnarray}\label{rrgf}&&~~~dim~\text{Ext}^2(F/Q'_1,Q'_1)=dim~\text{Hom}(Q'_1,\widetilde{F}_1(K_X))\nonumber\\
&\leq &dim~\text{Hom}(Q'_1,K_1(K_X+C))+dim~\text{Hom}(Q'_1,Q'_1(K_X+C))\nonumber\\
\leq &4g_C-2&+\chi(K_1)-\chi(Q'_1)<4g_C-2+\frac{(m_1+m_2-1)L^2}{(m_1+2m_2)^2}=:N.
\end{eqnarray} 

Now combine (\ref{time}) and (\ref{rrgf}) and we get an analogous formula to (\ref{exo}) as follows.  
\begin{eqnarray}\label{exi}&&dim~\mc_{\frac Lk}^{2,\cdots,2,1,\cdots,1}\cap\mt_{2}^o(L,\chi)\leq dim~\mathcal{P}_{F/Q'_1}+g_C-1+N-\chi(F/Q'_1,Q'_1)\nonumber\\
&\leq&(\frac {(m_1+2m_2-1)L}{m_1+2m_2})^2+\frac {(m_1+2m_2-2)L.K_X}{m_1+2m_2}\nonumber\\&&+g_C-1+N+\frac{(m_1+2m_2-1)L^2}{(m_1+2m_2)^2}\nonumber\\
&=& L^2+(K_X.L+1)+(1-g_C)+\frac{L.K_X}{2m-1}-\frac{(m_2-2)L^2}{(2m-1)^2}+1\nonumber\\&\leq& L^2+(K_X.L+1)~for~m_2\geq2.
\end{eqnarray}

If $m_2=1$, then $F/P_1\in\mc_{\frac Lk,a}^{1,\cdots,1}$ and by Lemma \ref{rko} the parametrizing stack $\mathcal{P}_{F/P_1}$ has dimension $\leq (\frac {m_1L}{m_1+2})^2+(\frac{m_1L.K_X}{m_1+2}+1)$.  On the other hand, $\text{Hom}(P_1,F/P_1(K_X))=\text{Hom}(P_1,S'_1(K_X+C))=\text{Hom}(P_1,Q'_1(K_X+C)).$

By (\ref{kqro}), (\ref{pot}) and the stability of $F$ we have
\begin{equation}\label{gcrt}\chi(Q'_1)+\chi(K_1)=\chi(R''_{1})-\frac{(m_1+1)L^2}{(m_1+2)^2}> \frac{2\chi}{m_1+2}-\frac{(m_1+1)L^2}{(m_1+2)^2}.\end{equation}

On the other hand, $Q'_1$ is a subsheaf of $F$.  Hence $\chi(Q'_1)\leq \frac{\chi}{(m_1+2)}$, then by (\ref{gcrt}) we have 
\begin{equation}\label{ubk}\chi(K_1)> \frac{\chi}{m_1+2}-\frac{(m_1+1)L^2}{(m_1+2)^2}.
\end{equation}
Hence $\chi(Q'_1)-\chi(K_1)<\frac{m_1+1}{(m_1+2)^2}L^2$.  Therefore
\begin{eqnarray}\label{grrg}&&dim~\text{Ext}^2(F/P_1,P_1)\leq dim~\text{Hom}(P_1,Q'_1(K_X+C))\nonumber\\
&\leq&dim~\text{Hom}(Q'_1,Q'_1(C+K_X))+dim~\text{Hom}(K_1(C),Q'_1(C+K_X))\nonumber\\
&\leq& 4g_C-2+\chi(Q'_1)-\chi(K_1(C))\nonumber\\
&\leq& 4g_C-2+\frac{m_1L^2}{(m_1+2)^2}=:N.
\end{eqnarray} 
Proposition \ref{coha} gives a upper bound for the dimension of the parametrizing stack of $P_1$.  By using analogous estimate to (\ref{exo}), we have
\begin{eqnarray}\label{eff}&&dim~\mc^{2,1,\cdots,1}_{\frac Lk}\cap\mt_2^o(L,\chi)\leq dim~\mathcal{P}_{F/P_1}+dim~\mathcal{P}_{P_1}+N-\chi(F/P_1,P_1)\nonumber\\
&\leq&(\frac {L}{m_1+2})^2(m_1^2+3m_1+5\frac12)+K_X.L+1+\frac{3K_X.L}{2(m_1+2)}+3\nonumber\\
&=& L^2+(K_X.L+1)+2(\frac{K_X.L}{m_1+2}+1)+(1-g_C)-\frac{(m_1-2)L^2}{(m_1+2)^2}\nonumber\\&\leq& L^2+(K_X.L+1)~for~m_1\geq2.
\end{eqnarray}
We proved the case $m_2=2$. 

The theorem is proved.  
\end{proof}

Yao YUAN \\
Yau Mathematical Sciences Center, Tsinghua University, \\
Beijing 100084, China\\
E-mail: yyuan@mail.math.tsinghua.edu.cn.

\end{document}